\documentclass{amsart}
\usepackage{Lie-alg-mu-FT}

\title[Fourier transforms of orbital integrals]
{Fourier transforms of semisimple orbital integrals on
the Lie algebra of $\protect\SL_2$}
\author{Loren Spice}
\email{l.spice@tcu.edu}
\address{%
Department of Mathematics \\
Texas Christian University \\
TCU Box 298900 \\
2840 W.\ Bowie St \\
Fort Worth, TX 76109%
}
\subjclass[2000]{Primary 22E50, 22E35}
\keywords{%
\(p\)-adic group, orbital integral, special function%
}
\begin{document}

\begin{abstract}
The Harish-Chandra--Howe local character expansion expresses
the characters of reductive, \(p\)-adic groups in terms of
Fourier transforms of nilpotent orbital integrals on their
Lie algebras,
and
Murnaghan--Kirillov theory expresses many characters of
reductive, \(p\)-adic groups in terms of Fourier transforms
of semisimple orbital integrals (also on their Lie algebras).
In many cases, the evaluation of these Fourier transforms
seems intractable;
but, for \(\SL_2\), the nilpotent orbital integrals
have already been computed
\cite{debacker-sally:germs}*{Appendix A}.
In this paper, we
use a variant of Huntsinger's integral formula,
and the theory of \(p\)-adic special functions,
to compute semisimple orbital integrals.
\end{abstract}

\maketitle

\setcounter{tocdepth}2
\tableofcontents

\section{Introduction}

\subsection{History}
\label{sec:history}

Harish-Chandra's \(p\)-adic Lefschetz principle suggests
that results in real harmonic analysis should have analogues
in \(p\)-adic harmonic analysis.  This principle has had too
many successes to list, but it is interesting
that the paths to results in the Archimedean and
non-Archimedean settings are often different.
One striking manifestation of this is that the
characters for the discrete series of
real groups were found \emph{before} the representations to
which they were associated were constructed
(see \cite{hc:disc-series-2}*{Theorem 16}
and \cite{schmid:homogeneous}*{Theorem 4});
whereas, in the \(p\)-adic setting, although we now have
explicit constructions of many representations
(see \cites{
	adler:thesis,
	bushnell-kutzko:gln-book,
	bushnell-kutzko:sln-1,
	bushnell-kutzko:sln-2,
	corwin:division-alg-all,
	howe:div-alg,
	morris:classical-2,
	moy:thesis,
	moy-prasad:k-types,
	stevens:classical-sc,
	yu:supercuspidal
}, among many others),
explicit character tables are still very rare.

This scarcity is of particular concern because,
as suggested by Sally, it should be the case that
``characters tell all''
\cite{sally-spice:characters}*{p.~104}.
Note, for example,
the recent work of Langlands \cite{langlands:transfert},
which uses in a crucial way
(see \S1.d \loccit) the character formul\ae\ of
\cite{sally-shalika:characters}
to show the existence,
but only for \(\SL_2\),
of a transfer map dual to the transfer
of stable characters.
It seems likely that one of the main obstacles to
extending the results of \cite{langlands:transfert}
to other groups
is the absence of explicit character formul\ae\ for them.

The good news here is that much \emph{is} known about
the behaviour of characters in general.
For example, the Harish-Chandra--Howe local character
expansion \cites{
	hc:queens,
	howe:two-conj,
	debacker:homogeneity
} and Murnaghan--Kirillov theory \cites{
	jkim-murnaghan:charexp,
	jkim-murnaghan:gamma-asymptotic,
	murnaghan:chars-u3,
	murnaghan:chars-sln,
	murnaghan:chars-classical,
	murnaghan:hc,
	murnaghan:chars-gln
} give information about the asymptotics
(near the identity element) of characters of \(p\)-adic
groups in terms of Fourier transforms of orbital integrals
(nilpotent or semisimple) on the Lie algebra,
and many existing character formul\ae\ are stated in terms
of such orbital integrals (see, for example,
\cite{debacker:thesis}*{Theorem 5.3.2},
\cite{spice:thesis}*{Theorems 6.6 and 7.18},
\cite{debacker-reeder:depth-zero-sc}*{Lemma 10.0.4},
and
\cite{adler-spice:explicit-chars}*{Theorem 7.1}).
See also \cite{adler-spice:explicit-chars}*{\S0.1}
for a more exhaustive description of what is known
in the supercuspidal case.

The bad news is that many applications require completely
explicit character tables---in particular, the evaluation of
Fourier transforms of orbital integrals when they appear---%
but that Hales \cite{hales:characters} has shown that
the orbital integrals may themselves be
`non-elementary'.
This term has a technical meaning, but, for our purposes,
it suffices to regard it informally as meaning
`difficult to evaluate'.
(Note, though, that the asymptotic behaviour of orbital
integrals `near \(\infty\)' is understood in all cases; see
\cite{waldspurger:loc-trace-form}*{Proposition VIII.1}.)
Since \(\SL_2\) is both simple enough for
many explicit computations to be tractable
(for example, the Fourier transforms of nilpotent orbital
integrals have already been computed,
in \cite{debacker-sally:germs}*{Appendix A.3--A.4}),
and complicated enough for
interesting phenomena to be apparent
(for example, unlike \(\GL_2\) and
	\(\operatorname{PGL}_2\),
it admits non-stable characters),
it is a natural focus for our investigations.

Another perspective on the behaviour of characters in
the range where Murnaghan--Kirillov theory holds
is offered in
\cite{corwin-moy-sally:gll}*{Theorem 4.2(d)},
\cite{takahashi:gl3}*{Proposition 2.9(2)},
and
\cite{takahashi:gll}*{Theorem 2.5},
where explicit mention of orbital integrals is replaced
(on the `bad shell'---see \S\ref{sec:bad-ram})
by arithmetically interesting sums,
identified in \cites{takahashi:gl3,takahashi:gll} as
Kloosterman sums.
In fact, exponential sums%
---specifically, Gauss sums---%
have long been observed in \(p\)-adic harmonic analysis;
see, for example,
\cite{shalika:thesis}*{\S1.3},
\cite{waldspurger:loc-trace-form}*{\S VIII.1},
\cite{debacker:thesis}*{p.~55},
\cite{corwin-moy-sally:gll}*{Proposition 3.7},
and
\cite{adler-spice:explicit-chars}*{\S5.2}.

The work recorded here was carried out while
preparing \cite{adler-debacker-sally-spice:sl2-chars}, which
provides a proof of the aforementioned
\(\SL_2\) character formul\ae\ \cite{sally-shalika:characters}
by specialising the results of
\cites{
	debacker-reeder:depth-zero-sc,
	adler-spice:explicit-chars
}.  As discussed above, these general results are stated
in terms of Fourier transforms of orbital integrals
(see Definition \ref{defn:mu-hat}); so, in order to obtain
completely explicit formul\ae, it was necessary to evaluate
those Fourier transforms.
The author of the present paper was
surprised to discover that this latter evaluation
reduced to the computation of \term{Bessel functions}
(see \S\ref{sec:Bessel} and Proposition \ref{prop:second-orbital});
but, in retrospect, by the aforementioned \(p\)-adic
Lefschetz principle,
it seems natural that the `special functions'
described in \cite{sally-taibleson:special}
will play some important role in
\(p\)-adic harmonic analysis,
since their classical analogues are so integral to
real harmonic analysis
(see, for just one example,
\cite{gindikin-karpelevic:plancherel}*{Theorem 2},
where Harish-Chandra's \(\mathbf c\)-function is
calculated in terms of \(\Gamma\)-functions).
Relationships between a different sort of Bessel function,
and a different sort of orbital integral (adapted to the
Jacquet--Ye relative trace formula),
have already been demonstrated by Baruch
\cites{
	baruch:bessel-gl2,
	baruch:bessel-qs,
	baruch:bessel-gln,
	baruch:bessel-fcn-gl3,
	baruch:bessel-dist-gl3
}.
We will investigate further applications of
complex-valued
\(p\)-adic special functions in future work.

\subsection{Outline of the paper}

In order that everything be completely explicit, we need to
carry around a large amount of notation;
we describe it in \S\S\ref{sec:notn}--\ref{sec:Bessel}.
Specifically, \S\ref{sec:notn}--\ref{sec:tori}
describe the basic notation
for working with groups over \(p\)-adic fields,
adapted to the particular setting of the group \(\SL_2\).
Since our formul\ae\ will be written `torus-by-torus'
(\textit{a la} Theorem 12 of \cite{hc:harmonic}),
we need to describe the tori in \(\SL_2\).
This can be done very concretely; see
Definition \ref{defn:tori}.

In \S\ref{sec:orbital}, we define the functions
\(\hat\mu^G_{X^*}\)
(Fourier transforms of orbital integrals)
that we want to compute
as representing functions for certain invariant
distributions on \(\sl_2\)
(see Definition \ref{defn:mu-hat} and
Notation \ref{notn:mu-hat}).
Since these functions are defined only up to scalar
multiples, it is important to be aware of the normalisations
involved in their construction.
In this respect, note that we specify the
(Haar) measures that we are using in
Definition \ref{defn:k-Haar}
and
Proposition \ref{prop:orbital-integral-integral}.

As mentioned in \S\ref{sec:history},
\(p\)-adic harmonic analysis tends to involve Gauss sums and
other fourth roots of unity, and our calculations are no
exception; we define and compare some of the relevant
constants in \S\ref{sec:roots}.
Finally, with these ingredients in place, we can
follow \cite{sally-taibleson:special} in defining the Bessel
functions that we will use to evaluate \(\hat\mu^G_{X^*}\).
Already, \cite{sally-taibleson:special} offers considerable
information about the values of these functions, but we need
to carry the calculations further,
especially far from the identity
(see Proposition \ref{prop:Bessel-shallow})
and on the `bad shell'
(see Proposition \ref{prop:Bessel-Kloost})%
---where (twisted) Kloosterman sums make an appearance.

In \S\ref{sec:mock-mu}, we define a function
\(M^G_{X^*}\) (see Definition \ref{defn:mock-mu}),
which we will spend most of the rest of the paper computing.
This is a reasonable focus because, once the computations are
completed, Proposition \ref{prop:orbital-integral-integral}
will show that we have actually been computing
\(\hat\mu^G_{X^*}\).
The definition of \(M^G_{X^*}\) involves a rather remarkable
function \(\varphi_\theta\)
(see Definition \ref{defn:sl2-as-fields} and
Lemma \ref{lem:torus-acts});
it seems likely that generalising our techniques will require
understanding the proper replacement for \(\varphi_\theta\).

Proposition \ref{prop:second-orbital} describes \(M^G_{X^*}\)
in terms of Bessel functions,
and
Proposition \ref{prop:Bessels} uses
Theorem \ref{thm:sally-taibleson:special:bessel}
to describe their behaviour near \(0\).

We now proceed according to the `type' of \(X^*\)
(as in Definition \ref{defn:split-un-or-ram}).
The calculations when \(X^*\) is split, and when it is
unramified, are quite similar; we combine them in
\S\ref{sec:orbital-spun}.
We split into cases depending on whether the argument
to \(M^G_{X^*}\) is far from
(as in \S\ref{sec:shallow-spun})
or close to
(as in \S\ref{sec:close-spun})
zero; there are qualitative differences in the behaviour,
as can be seen by comparing, for example,
Theorems \ref{thm:vanish-spun} and \ref{thm:close-spun}.
When \(X^*\) is ramified, it turns out that,
in addition to the behaviour far from
(as in \S\ref{sec:shallow-ram})
and close to
(as in \S\ref{sec:close-ram})
zero, there is a third range of interest in the middle.
This is the so called `bad shell'
(see \S\ref{sec:bad-ram}), and it seems likely that the
particularly complicated nature of the formul\ae\ here is a
reflection of the
`non-elementary' behaviour of orbital integrals
(hence, by Murnaghan--Kirillov theory, also of characters)
described in \cite{hales:characters}.

Finally, we show in \S\ref{sec:orbital-redux}
that the function that we have been evaluating actually does
represent the desired distribution, i.e., is equal to
\(\hat\mu^G_{X^*}\).
(See Proposition \ref{prop:orbital-integral-integral}.)
We close with some observations
(see Theorem \ref{thm:uniform})
about the qualitative behaviour of orbital integrals that
does not depend (much) on the `type' of \(X^*\).

\subsection{Acknowledgements}

This paper, and the paper
\cite{adler-debacker-sally-spice:sl2-chars} 
that follows it, would not have been possible without the
advice and guidance of Paul J. Sally, Jr.
It is a pleasure to thank him,
as well as Stephen DeBacker and Jeffrey D. Adler,
both of whom offered useful suggestions regarding this
paper.

The author was partially supported by NSF award
DMS-0854897.

\section{Notation}
\label{sec:notn}

Suppose that \indexmem\field is
a non-discrete, non-Archimedean local field.
We do not make any assumptions on its characteristic,
but we assume that its residual characteristic \indexmem p
is not \(2\).
(We occasionally cite \cite{shalika:thesis}, which works only
with characteristic-\(0\) fields; but we shall not use any
results from there that require this restriction.)
Let \indexmem\pint denote the ring of integers in \(\field\),
\indexmem\pp the prime ideal of \(\pint\),
and \indexmem\ord the valuation on \(\field\)
with value group \(\Z\).

Let \resfld denote the residue field \(\pint/\pp\) of \(\field\).
We write \(\indexmem q = \card\resfld\) for
the number of elements in \resfld,
and put \(\indexmem{\abs x} = q^{-\ord(x)}\)
for \(x \in \field\).
If \(\alpha \in \C\), then
we will write \indexmem{\nu^\alpha} for the
(multiplicative) character \(x \mapsto \abs x^\alpha\) of
\(\field\mult\).

Put \(\indexmem\bG = \SL_2\) and \(\indexmem G = \bG(k)\), and
let \indexmem\gg and \indexmem{\gg^*} denote the
Lie algebra and dual Lie algebra of \(G\), respectively.

It is important for our calculations to be quite specific
about the Haar measures that we are using.  For convenience,
we fix the ones used in \cite{sally-taibleson:special} (see
p.~280 \loccit).

\begin{defn}
\label{defn:k-Haar}
Throughout, we shall use the (additive) Haar measure
\(\textup dx\) on \(\field\) that assigns measure \(1\)
to \(\pint\), and the associated (multiplicative) Haar measure
\(\textup d\mult x = \abs x\inv\textup dx\) on
\(\field\mult\) that assigns measure \(1 - q\inv\) to
\(\pint\mult\).
When convenient, we shall write \(\textup dt\) instead of
\(\textup dx\).
\end{defn}

\begin{defn}
\label{defn:k-chars}
If \(\Phi\) is an (additive) character of \(\field\), then
we define \(\Phi_b : x \mapsto \Phi(b x)\) for
\(b \in \field\).
The \term{depth} of \(\Phi\) is
\[
\depth(\Phi)
\ldef \min \sett{i \in \Z}
	{\(\Phi\) is trivial on \(\pp^{i + 1}\)}
\]
(if \(\Phi\) is non-trivial) and \(\depth(\Phi) = {-}\infty\)
otherwise.
\end{defn}

The depth of a character is related to
what is often called its \term{conductor} by
\(\depth(\Phi) = \omega(\Phi) - 1\)
(in the notation of \cite{shalika:thesis}*{\S1.3}).
We have that
\begin{equation}
\label{eq:depth-Phi-b}
\depth(\Phi_b) = \depth(\Phi) - \ord(b).
\end{equation}
Note that the notion of depth, and the symbol \(\depth\),
will be used in multiple contexts
(see Definition \ref{defn:filt-and-depth});
we rely on the context to disambiguate them.

\begin{notn}
\label{notn:Phi}
\(\Phi\) is a non-trivial (additive) character of
\(\field\).
\end{notn}

One of the crucial tools of Harish-Chandra's approach to
harmonic analysis is the reduction, whenever possible, of
questions about a group to questions about its Lie algebra.
The exponential map often allows one to effect this
reduction, but, since it might converge only in a very small
neighbourhood of \(0\), we replace it with a
`mock-exponential map' (see \cite{adler:thesis}*{\S1.5})
which has many of the same properties
(see Lemma \ref{lem:cayley-facts}).

\begin{defn}
\label{defn:cayley}
The \term{Cayley map}
\(\cayley : \field \setminus \sset1
\to \field \setminus \sset{-1}\)
is defined by
\[
\cayley(X) = (1 + X)(1 - X)\inv
\quad\text{for \(X \in \field \setminus \sset1\)}.
\]
\end{defn}

The Cayley function is available in many settings; note that
we are using it only as a function defined almost everywhere
on \(\field\).  We gather a few of its properties below.

\begin{lemma}
\label{lem:cayley-facts}
\hfill\begin{itemize}
\item
The map \(\cayley\) is a bijection.
\item
\(\cayley(-X) = \cayley(X)\inv = \cayley\inv(X)\)
for \(X \in \field \setminus \sset{\pm1}\).
\item
The map \(\cayley\) carries
\(\pp^i\) to \(1 + \pp^i\) for all \(i \in \Z_{> 0}\).
\item
In the notation of Definition \ref{defn:k-Haar},
the pull-back along \(\cayley\) of the measure
\(\textup d\mult x\) on \(1 + \pp\) is the measure
\(\textup dx\) on \(\pp\).
\item
If \(X \in \pp^i\) and \(Y \in \pp^j\),
with \(i, j \in \Z_{> 0}\),
then
\[
\cayley(X + Y) \equiv \cayley(X) + 2Y
	\pmod{1 + \pp^n},
\]
where \(n = j + \min \sset{2i, j}\).
\end{itemize}
\end{lemma}

\begin{proof}
It is easy to check that
\(x \mapsto (1 - x)(1 + x)\inv\) is inverse to \(\cayley\)
and satisfies the desired equalities,
and that
\(\cayley(\pp^i) \subseteq 1 + \pp^i\)
and
\(\cayley\inv(1 + \pp^i) \subseteq \pp^i\).
If \(f \in C^\infty(1 + \pp)\), then there is some
\(i \in \Z_{> 0}\) such that
\(f \in C(1 + \pp/1 + \pp^i)\).  Upon noting that
\(\meas_{\textup dx}(\pp^i)
= q^{-i}
= \meas_{\textup d\mult x}(1 + \pp^i)\),
we see that
\begin{multline*}
\int_{1 + \pp} f(x)\textup d\mult x
= \sum_{x \in 1 + \pp/1 + \pp^i}
	f(x)\meas_{\textup d\mult x}(1 + \pp^i) \\
= \sum_{x \in \pp/\pp^i}
	(f \circ \cayley)(x)q^{-i}\meas_{\textup dx}(\pp^i)
= \int_\pp (f \circ \cayley)(x)\textup dx.
\end{multline*}

Finally, under the stated conditions on \(X\) and \(Y\),
\begin{multline*}
\bigl(\cayley(X) + 2Y\bigr)
\bigl(1 - (X + Y)\bigr) \\
= \cayley(X)\dotm(1 - X) +
Y\bigl(2(1 - (X + Y)) - \cayley(X)\bigr) \\
= (1 + X + Y) +
Y\bigl((1 - 2X - \cayley(X)) - 2Y\bigr).
\end{multline*}
Since \(\cayley(X) = 1 + 2X(1 - X)\inv\),
we have that
\(1 - 2X - \cayley(X) \in \pp^{2i}\).
The result follows.
\end{proof}

\section{Fields and algebras}
\label{sec:sqrt}

\begin{defn}
\label{defn:sqrt}
For \(\theta \in \field\mult\), we write
\(\field_\theta\) for the \(\field\)-algebra that is
\(\field \oplus \field\) (as a vector space), equipped with
multiplication
\((a, b)\dotm(c, d)
= (a c + b d\theta, a d + b c)\).
We write \(\sqrt\theta\) for the element
\((0, 1) \in \field_\theta\), so that
\((a, b) = a + b\sqrt\theta\).
\end{defn}

We also use the notation \(\sqrt\theta\) for a matrix (see
Definition \ref{defn:tori}); we shall rely on context to
make the meaning clear.

If \(\theta \not\in (\field\mult)^2\), then
\(\field_\theta\) is isomorphic to \(\field(\sqrt\theta)\)
(as \(\field\)-algebras)
via the map \((a, b) \mapsto a + b\sqrt\theta\), and we
shall not distinguish between them.

If \(\theta = x^2\), with \(x \in \field\), then
\(\field_\theta\) is isomorphic to \(\field \oplus \field\)
(as \(\field\)-algebras)
via the map \((a, b) \mapsto (a + b x, a - b x)\).

\begin{defn}
\label{defn:maps-and-gps}
We define
\begin{gather*}
\Norm_\theta(a + b\sqrt\theta) = a^2 - b^2\theta,
\quad
\Tr_\theta(a + b\sqrt\theta) = 2a, \\
\Re_\theta(a + b\sqrt\theta) = a,
\quad
\Im_\theta(a + b\sqrt\theta) = b, \\
\intertext{and}
\ord_\theta(a + b\sqrt\theta)
= \tfrac1 2\ord\bigl(\Norm_\theta(a + b\sqrt\theta)\bigr)
\end{gather*}
for \(a + b\sqrt\theta \in \field_\theta\).
Write
\(C_\theta = \ker \Norm_\theta\)
and
\(V_\theta = \ker \Tr_\theta\),
and
let \(\sgn_\theta\) be the unique (multiplicative) character
of \(\field\mult\) with kernel precisely
\(\Norm_\theta(\field_\theta\mult)\).
\end{defn}

If \(\theta \not\in (\field\mult)^2\), then
\(\Norm_\theta\) and \(\Tr_\theta\) are the usual
norm and trace maps associated to the quadratic extension of
fields \(\field_\theta/\field\),
and \(\ord_\theta\) is the valuation on \(\field_\theta\)
extending \(\ord\).
In any case,
\(\field_\theta\mult
= \set{z \in \field_\theta}{\Norm_\theta(z) \ne 0}\).

We can describe the signum character explicitly by
\begin{equation}
\label{eq:sgn-spun}
\sgn_\theta(x) = \begin{cases}
1,              & \text{\(\theta\) split}       \\
(-1)^{\ord(x)}, & \text{\(\theta\) unramified,}
\end{cases}
\end{equation}
and
\begin{equation}
\label{eq:sgn-ram}
\begin{aligned}
\sgn_\theta(\theta) & {}= \sgn_\resfld(-1) \\
\sgn_\theta(x)      & {}= \sgn_\resfld(\ol x)
	\quad\text{for \(x \in \pint\mult\),}
\end{aligned}
\end{equation}
where \(\sgn_\resfld\) is the quadratic character of
\(\resfld\mult\) and \(x \mapsto \ol x\) the reduction map
\(\pint \to \resfld\).

\section{Tori and filtrations}
\label{sec:tori}

We begin by defining a few model tori.

\begin{defn}
\label{defn:tori}
For \(\theta \in \field\), put
\[
\bT_\theta
= \set{\begin{pmatrix}
a       & b \\
b\theta & a
\end{pmatrix}}{a^2 - b^2\theta = 1}.
\]
Then
\[
\pmb\ttt_\theta \ldef \Lie(\bT_\theta)
= \sset{\begin{pmatrix}
0       & b \\
b\theta & 0
\end{pmatrix}}.
\]
We write \(\sqrt\theta\) for the element
\(\begin{smallpmatrix}
0      & 1 \\
\theta & 0
\end{smallpmatrix} \in \ttt_\theta\), so that
\(\ttt_\theta = \operatorname{Span}_k \sqrt\theta\).
We will call a maximal \(\field\)-torus in \bG
\term{standard} exactly when it is of the form
\(\bT_\theta\) for some \(\theta \in \field\).
\end{defn}

We also use the notation \(\sqrt\theta\) for an element of
an extension of \(\field\) (see Definition \ref{defn:tori}); we
shall rely on context to make the meaning clear.

\begin{rem}
The group
\(T_\theta\) is isomorphic to
\(C_\theta = \ker \Norm_\theta\),
and the Lie algebra \(\ttt_\theta\) to
\(V_\theta = \ker \Tr_\theta\),
in each case via the map
\(\begin{smallpmatrix}
a       & b \\
b\theta & a
\end{smallpmatrix} \mapsto (a, b)\).
\end{rem}

We shall use the terms `split', `unramified', and `ramified'
in many different contexts.

\begin{rem}
If \bT is a maximal \(\field\)-torus in \bG
and \(\ttt = \Lie(T)\),
then we shall identify
\(\ttt\) (respectively, \(\ttt^*\))
with the spaces of fixed points for the
adjoint (respectively, co-adjoint)
action on
\(\gg\) (respectively, \(\gg^*\)).
By abuse of language, we shall sometimes say that
\(X^* \in \gg^*\) or \(Y \in \gg\)
lies in, or belongs to, the torus \bT
to mean that \(X^* \in \ttt^*\) and \(Y \in \ttt\);
equivalently, that \(C_\bG(X^*) = \bT = C_\bG(Y)\).
In particular,
``\(X^*\) and \(Y\) belong to a common torus''
is shorthand for
``\(C_\bG(X^*) = C_\bG(Y)\)''.
\end{rem}

\begin{defn}
\label{defn:split-un-or-ram}
A maximal \(\field\)-torus in \bG is called
(un)ramified according as it is elliptic and splits over an
(un)ramified extension of \(\field\).
An element \(\theta \in \field\) is called split,
unramified, or ramified according as
\(\bT_\theta\) has that property.
A regular, semisimple element of \(\gg\) or \(\gg^*\) is
called split, unramified, or ramified according as the torus
to which it belongs has that property.
\end{defn}

\begin{rem}
To be explicit, squares in \(\field\mult\) are split,
and a non-square \(\theta \in \field\) is
unramified or ramified according as
\(\max \set{\ord(x^2\theta)}{x \in \field}\)
is even or odd, respectively.
\end{rem}

\begin{notn}
\label{notn:Weyl}
If \bT is a maximal \(\field\)-torus in \bG,
with \(T = \bT(\field)\),
then we write
\(W(\bG, \bT) = N_\bG(\bT)/\bT\) for the absolute,
and
\(W(G, T) = N_G(T)/T\) for the relative,
Weyl group of \bT in \bG.
\end{notn}

Every maximal \(\field\)-torus in \bG is \(G\)-conjugate to
some \(\bT_\theta\).
(See, for example, \cite{debacker-sally:germs}*{\S A.2}.)
In particular,
\[
\Int\begin{pmatrix}
1    & 1   \\
-1/2 & 1/2
\end{pmatrix}\set{\begin{pmatrix}
a & 0 \\
0 & d
\end{pmatrix}}{a d = 1} = \bT_1.
\]

\begin{rem}
\label{rem:Weyl}
For all \(\theta \in \field\),
the group \(W(\bG, \bT_\theta)\) has order \(2\), with the
non-trivial element acting on \(\bT_\theta\) by inversion.
If \(\sgn_\theta(-1) = 1\)
(in particular, if \(\theta\) is split or unramified),
say, with \(\Norm_\theta(a + b\sqrt\theta) = -1\),
then \(W(G, T_\theta)\) also has order \(2\), with the
non-trivial element represented by
\(\begin{smallpmatrix}
a        & b   \\
-b\theta & -a
\end{smallpmatrix}\).
If \(\theta = 1\), then we may take \((a, b) = (0, 1)\)
to recover the familiar Weyl-group element.
Otherwise (i.e., if \(\sgn_\theta(-1) = -1\)),
\(W(G, T_\theta)\) is trivial.
\end{rem}

The concept of \term{stable conjugacy} was introduced by
Langlands as part of the foundation of the Langlands
conjectures;
see \cite{langlands:stable}*{pp.~2--3}.

\begin{defn}
Two
\begin{itemize}
\item maximal \(\field\)-tori \(\bT_i\) in \bG,
\item regular semisimple elements \(X^*_i \in \gg^*\),
or
\item regular semisimple elements \(Y_i \in \gg\),
\end{itemize}
with \(i = 1, 2\), are called \term{stably conjugate}
exactly when there are a field extension \(E/\field\) and an
element \(g \in \bG(E)\) such that
\begin{itemize}
\item \(\Int(g)T_1 = T_2\)
or
\item \(\Ad^*(g)X^*_1 = X^*_2\)
or
\item \(\Ad(g)X_1 = X_2\),
\end{itemize}
where \(T_i = \bT_i(\field)\) for \(i = 1, 2\).
If the conjugacy can be carried out without passing to an
extension field (i.e., if we may take \(g \in G\)), then we
will sometimes emphasise this by saying that the tori or
elements are \term{rationally conjugate}.
\end{defn}

Note that the Zariski-density of \(T_i\) in \(\bT_i\)
implies that \(\Int(g)\bT_1 = \bT_2\), but that this is a
strictly weaker condition; indeed, given \emph{any} two
maximal tori, there is an element \(g\), defined over some
extension field of \(\field\), satisfying this condition.
In our special case (of \(\bG = \SL_2\)), we have that
two tori or elements are stably conjugate if and only if
they are conjugate in \(\GL_2(\field)\).

More concretely,
two tori \(\bT_\theta\) and \(\bT_{\theta'}\) are stably
conjugate if and only if
\(\theta \equiv \theta' \pmod{(\field\mult)^2}\).
The stable conjugacy class of the split torus \(\bT_1\)
is also a rational conjugacy class.

Suppose that
\(\epsilon\) is an unramified, and \(\varpi\) a ramified,
non-square.
Then the stable conjugacy class of \(\bT_\epsilon\) splits
into 2 rational conjugacy classes, represented by
\(\bT_\epsilon\) and \(\bT_{\varpi^2\epsilon}\).
The stable conjugacy class of \(\bT_\varpi\) is also a
rational conjugacy class if \(\sgn_\varpi(-1) = -1\);
but it splits into 2 rational conjugacy classes,
represented by \(\bT_\varpi\) and \(\bT_{\epsilon^2\varpi}\),
if \(\sgn_\varpi(-1) = 1\).

We also need filtrations on the Lie algebra, and dual Lie
algebra, of a torus.
These definitions are standard
(see, for example, \cite{adler:thesis}*{\S1.4})
and can be made in far more generality
(see \cite{moy-prasad:k-types}*{\S3}
and \cite{moy-prasad:jacquet}*{\S3.3});
we give only simple definitions adapted to \(\bG = \SL_2\).

\begin{defn}
\label{defn:filt-and-depth}
Let \bT be a maximal \(\field\)-torus in \bG, and put
\(\ttt = \Lie(\bT(k))\).
Recall that \bT is \(G\)-conjugate to \(\bT_\theta\) for
some \(\theta \in \field\),
so that
\(\ttt = \Lie(T)\) is isomorphic to
\(V_\theta = \ker \Tr_\theta \subseteq \field_\theta\).
For \(r \in \R\), we write
\(\ttt_r\) for the pre-image of
\(\set{Y \in V_\theta}{\ord_\theta(Y) \ge r}\)
and
\(\ttt_{r{+}}\) for the pre-image of
\(\set{Y \in V_\theta}{\ord_\theta(Y) > r}\);
and then we write
\(\ttt^*_r = \sett{X^* \in \ttt^*}{
	\(\Phi(\langle X^*, Y\rangle) = 1\)
	for all \(Y \in \ttt_{(-r){+}}\)
}\)
(where \(\Phi\) is the additive character of
Notation \ref{notn:Phi}).

If \(X^* \in \ttt^*\) and \(Y \in \ttt\),
then we define
\(\depth(X^*) = \max \set{r \in \R}{X^* \in \ttt^*_r}\)
and
\(\depth(Y) = \max \set{r \in \R}{Y \in \ttt_r}\).
\end{defn}

One can define a notion of depth in more generality (see,
for example,
\cite{adler-debacker:bt-lie}*{\S3.3 and Example 3.4.6}
and
\cite{jkim-murnaghan:charexp}*{\S2.1 and Lemma 2.1.5}),
but we only need the special case above.
(The only remaining case to consider for
\(\gg = \sl_2(\field)\) is the depth of a nilpotent
element, which is \(\infty\).)

\section{Orbital integrals}
\label{sec:orbital}

Our goal in this paper is to compute Fourier transforms of
regular, semi-simple orbital integrals on \(\gg\) (see
Definition \ref{defn:mu-hat} below).
Since the Fourier transforms of nilpotent orbital integrals
were computed in \cite{debacker-sally:germs}*{Appendix A}, this
covers all Fourier transforms of orbital integrals on
\(\gg\) (for our particular case \(\bG = \SL_2\)).
The case of orbital integrals on \(G\) was discussed in
\cite{sally-shalika:orbital-integrals}, as the culmination
of the series of papers that began with
\cites{sally-shalika:characters,sally-shalika:plancherel}.

We will begin by choosing a representative for the regular,
semi-simple orbit of interest.
By \S\ref{sec:tori}, we may choose this representative in a
standard torus
(in the sense of Definition \ref{defn:tori}).

\begin{notn}
\label{notn:X*}
\(\beta, \theta \in \field\mult\),
and
\(X^* = \beta\dotm\sqrt\theta \in \ttt_\theta^*\).
\end{notn}

Here, we are implicitly using the identification of
\(\ttt_\theta\) with \(\ttt_\theta^*\) via the trace form;
what we really mean is that
\(\langle X^*, Y\rangle = \Tr \beta\dotm\sqrt\theta\dotm Y\)
for \(Y \in \ttt_\theta\),
where \(\langle\cdot, \cdot\rangle\) is the usual pairing
between \(\ttt_\theta^*\) and \(\ttt_\theta\).

As in Definition \ref{defn:k-chars},
we may define a new character \(\Phi_\beta\) of \(\field\).
This character will occur often enough in our calculations
that it is worthwhile to give it a name.

\begin{notn}
\label{notn:depth-and-scdepth}
\(-r = \depth(X^*)\),
\(\scPhi = \Phi_\beta\),
and
\(\scdepth = \depth(\scPhi)\).
\end{notn}

By Definition \ref{defn:filt-and-depth},
\(Y \mapsto \Phi(\langle X^*, Y\rangle)\)
is trivial on \((\ttt_\theta)_{r{+}}\), but not on
\((\ttt_\theta)_r\).
Therefore,
\(\scdepth = r + \tfrac1 2\ord(\theta)\).

Since \(C_G(X^*) = T_\theta\) is Abelian, it is unimodular;
so there exists a measure on \(G/C_G(X^*)\) invariant under
the action of \(G\) by left translation.

\begin{notn}
\label{notn:orbit-measure}
Let \(\textup d\dot g\) be a translation-invariant measure on
\(G/C_G(X^*)\).
\end{notn}

Since the orbit, \(\mc O^G_{X^*}\), of \(X^*\) under
the co-adjoint action of \(G\) is isomorphic as a \(G\)-set
to \(G/C_G(X^*)\), we could transport to it the measure on
the latter space; but we do not find it convenient to do so.

Since \(X^*\) is semisimple, \(\mc O^G_{X^*}\) is
closed in \(\gg^*\)
(see, for example, Proposition 34.3.2 of
\cite{tauvel-yu:lie-alg-and-alg-gps}).
Therefore, the restriction to
\(\mc O^G_{X^*}\) of a locally constant, compactly
supported function on \(\gg^*\) remains locally constant and
compactly supported, so that the following definition makes
sense.

\begin{defn}
\label{defn:mu}
The orbital integral of \(X^*\) is the distribution
\indexmem{\mu^G_{X^*}} on \(\gg^*\) defined by
\[
\mu^G_{X^*}(f^*)
= \int_{G/C_G(X^*)}
	f^*(\Ad^*(g)X^*)
\textup d\dot g
\quad\text{for all \(f^*\in C_c^\infty(\gg^*)\).}
\]
\end{defn}

We are interested in the Fourier transform of
\(\mu^G_{X^*}\).  The definition of the Fourier transform
(of distributions or of functions) requires, in addition to
a choice of additive character (see Notation \ref{notn:Phi}),
also a choice of Haar measure \(\textup dY\) on \(\gg^*\);
but we shall build this choice into our representing
function (see Notation \ref{notn:mu-hat}), so that it will
not show up in our final answer.

\begin{defn}
\label{defn:mu-hat}
The Fourier transform of the orbital integral of \(X^*\) is
the distribution \indexmem{\hat\mu^G_{X^*}} on \(\gg\)
defined for all \(f \in C_c^\infty(\gg)\) by
\[
\hat\mu^G_{X^*}(f) = \mu^G_{X^*}(\hat f),
\]
where
\[
\hat f(Y^*)
= \int_\gg f(Y)\Phi(\langle Y^*, Y\rangle)\textup dY
\quad\text{for all \(Y^* \in \gg^*\).}
\]
\end{defn}

It is a result of Harish-Chandra
(see \cite{hc:queens}*{Theorem 1.1}) that
\(\hat\mu^G_{X^*}\) is \emph{representable} on \(\gg\);
i.e., that there exists a locally integrable function \(F\)
on \(\gg\) such that
\[
\hat\mu^G_{X^*}(f)
= \int_G f(Y)F(Y)\textup dY
\quad\text{for all \(f \in C_c^\infty(\gg)\).}
\]
One can say more about the behaviour and asymptotics of the
function \(F\).
For example, it turns out that it blows up as \(Y\)
approaches \(0\), but that its blow-up is controlled by a
power of a discriminant function.

\begin{defn}
\label{defn:disc}
The \term{Weyl discriminant} on \(\gg\) is the
function \(\indexmem{D_\gg} : \gg \to \C\) such that,
for all \(Y \in \gg\),
\(D_\gg(Y)\) is the coefficient of the degree-\(1\) term
in the characteristic polynomial of \(\ad(Y)\).
Concretely,
\[
D_\gg\begin{pmatrix}
a & b  \\
c & -a
\end{pmatrix} = 4(a^2 + bc).
\]
\end{defn}

Our main interest, however, is in the restriction of the
function \(F\) above
to the set \(\gg\rss\) of regular, semisimple elements,
where it is locally constant.

\begin{notn}
\label{notn:mu-hat}
By abuse of notation, we write again \(\hat\mu^G_{X^*}\) for
the function that represents the restriction to \(\gg\rss\)
of \(\hat\mu^G_{X^*}\).
\end{notn}

When we refer to the computation of the Fourier transform of
an orbital integral, it is actually the (scalar) function of
Notation \ref{notn:mu-hat} that we are trying to compute.
The main tool in this direction is a general integral
formula of Huntsinger (see
\cite{adler-debacker:mk-theory}*{Theorem A.1.2}), but we
find it easier to evaluate an integral adapted to our
current setting (see Definition \ref{defn:mock-mu}).
The computation of this integral will occupy most of the
paper; once that is done, we shall finally prove that it
actually represents the distribution \(\hat\mu^G_{X^*}\)
(see Proposition \ref{prop:orbital-integral-integral}).

Finally, we fix an element at which to evaluate the
functions of interest.  Since \(\hat\mu^G_{X^*}\), as just
defined, and \(M^G_{X^*}\) below
(see Definition \ref{defn:mock-mu})
are \(G\)-invariant functions on \(\gg\rss\),
we may again consider only elements of standard tori.

\begin{notn}
\label{notn:Y}
\(s, \theta' \in \field\mult\),
and
\(Y = s\dotm\sqrt{\theta'} \in \ttt_{\theta'}\).
\end{notn}

Our computations will be phrased in terms of the values of
two `basic' functions at \(Y\).

\begin{lemma}
\label{lem:Y-facts}
\(\depth(Y) = \tfrac1 2\ord(s^2\theta')\)
and
\(D_\gg(Y) = 4s^2\theta'\).
\end{lemma}

\begin{proof}
This is a straightforward consequence of
Definitions \ref{defn:filt-and-depth} and \ref{defn:disc}.
\end{proof}

\section{Roots of unity and other constants}
\label{sec:roots}

The computation of Fourier transforms of orbital integrals
on \(\gg\)---hence, via
Murnaghan--Kirillov theory
\cites{adler-debacker:mk-theory,
adler-spice:explicit-chars,
jkim-murnaghan:charexp,
jkim-murnaghan:gamma-asymptotic,
murnaghan:chars-sln}, also of the values near the identity
of characters of \(G\)
(cf.~\cites{sally-shalika:characters,
adler-debacker-sally-spice:sl2-chars})---involves a
somewhat bewildering array of \(4\)th roots of unity, for
each of which there is a variety of notation available.
It turns out that all of these can be expressed in terms of
a single `basic' quantity, the Gauss sum, denoted by
\(G(\Phi)\) in \cite{shalika:thesis}*{Lemma 1.3.2}.
The definition there implicitly depends on a choice of
uniformiser, denoted there by \(\pi\).
Although the choice is arbitrary, we shall
find it convenient for later usage to denote it by
\(-\varpi\).
Recall from Notation \ref{notn:Phi}
that \(\Phi\) is a non-trivial (additive) character
of \(\field\).

\begin{defn}
\label{defn:G}
If \(\varpi\) is a uniformiser of \(\field\), then
\[
\indexmem{G_\varpi(\Phi)}
\ldef q^{-1/2}\sum_{X \in \pint/\pp}
	\Phi_{(-\varpi)^{\depth(\Phi)}}(X^2).
\]
\end{defn}

It is possible to compute these values exactly (see, for
example,
\cite{lidl-niederreiter:finite-fields}*{Theorem 5.15}), but
we shall only require a few transformation laws.

\begin{lemma}
\label{lem:G-facts}
If \(\varpi\) is a uniformiser of \(\field\), then
\begin{align*}
G_{b\varpi}(\Phi)
& {}= \sgn_\varpi(b)^{\depth(\Phi)}G_\varpi(\Phi)
	\quad\text{for \(b \in \pint\mult\),} \\
G_\varpi(\Phi_b) & {}= \sgn_\varpi(b)G_\varpi(\Phi)
	\quad\text{for \(b \in \field\mult\),} \\
G_\varpi(\Phi)^2 & {}= \sgn_\varpi(-1), \\
\intertext{and}
G_\varpi(\Phi)
& {}= q^{-1/2}\sgn_\varpi(-1)^{\depth(\Phi)}
\sum_{X \in \resfld\mult} \ol\Phi(X)\sgn_\resfld(X),
\end{align*}
where
\(\sgn_\resfld\) is the quadratic character of \(\resfld\mult\),
and
\(\ol\Phi\) the (additive) character of \(\resfld = \pint/\pp\)
arising from the restriction to \(\pint\) of the depth-\(0\)
character \(\Phi_{\varpi^{\depth(\Phi)}}\) of \(\field\).
\end{lemma}

\begin{proof}
Since \(\sum_{X \in \resfld} \ol\Phi(X) = 0\), we have that
\begin{align*}
\sum_{X \in \resfld\mult} \ol\Phi(X)\sgn_\resfld(X)
& {}= \ol\Phi(0) +
\sum_{X \in \resfld\mult}
	\ol\Phi(X)\bigl(1 + \sgn_\resfld(X)\bigr) \\
& {}= \ol\Phi(0) +
2\sum_{X \in (\resfld\mult)^2} \ol\Phi(X) \\
& {}= \sum_{X \in \resfld} \ol\Phi(X^2) \\
& {}= q^{1/2}G_\varpi(\Phi_{(-1)^{\depth(\Phi)}}).
\end{align*}
In other words,
\begin{equation}
\tag{$*$}
\label{eq:G-as-G}
G_\varpi(\Phi_{(-1)^{\depth(\Phi)}})
= q^{-1/2}G(\sgn_\resfld, \ol\Phi),
\end{equation}
where the notation on the right is as in
\cite{lidl-niederreiter:finite-fields}*{\S5.2}
(except that their \(\psi\) is our \(\sgn_\resfld\),
the quadratic character of \(\resfld\mult\),
and their \(\chi\) is our \(\ol\Phi\)).
The third equality, and the second equality for
\(b \in \pint\mult\), now follow from Theorem 5.12
\loccitthendot.
The first equality follows from the second upon noting that
\(G_{b\varpi}(\Phi) = G_\varpi(\Phi_{b^{\depth(\Phi)}})\);
and
taking \(b = (-1)^{\depth(\Phi)}\) and combining with
\eqref{eq:G-as-G} gives the fourth equality.
Finally, by definition,
\(G_\varpi(\Phi_{(-\varpi)^n})
= G_\varpi(\Phi) = \sgn_\varpi(-\varpi)^n G_\varpi(\Phi)\)
for all \(n \in \Z\).
\end{proof}

By Proposition \ref{prop:second-orbital} and
Theorem \ref{thm:sally-taibleson:special:bessel},
our calculations will involve the \(\Gamma\)-factors
defined in \cite{sally-taibleson:special}*{\S3}.
Of particular interest is \(\Gamma(\nu^{1/2}\sgn_\varpi)\).
By Theorem 3.1(iii) \loccit,
\(\Gamma(\nu^{1/2}\sgn_\varpi)^2 = \sgn_\varpi(-1)\), so
that, by Lemma \ref{lem:G-facts},
\(\Gamma(\nu^{1/2}\sgn_\varpi) = \pm G_\varpi(\Phi)\).
It will be useful to identify the sign.

\begin{lemma}
\label{lem:Gamma-varpi}
If \(\varpi\) is a uniformiser of \(\field\), then
\(\Gamma(\nu^{1/2}\sgn_\varpi)
= \sgn_\varpi(-1)^{\depth(\Phi) + 1}G_\varpi(\Phi)\).
\end{lemma}

\begin{proof}
Write \(\ol\Phi = \Phi_{\varpi^{\depth(\Phi)}}\);
this is a depth-\(0\) character of \(\field\).
The definitions of \cite{sally-taibleson:special} depend on
a depth-(\(-1)\) additive character \(\chi\); we take it to
be \(\ol\Phi_\varpi\).
The definition of \(\Gamma(\nu^{1/2}\sgn_\varpi)\) involves
a principal-value integral (see Definition
\ref{defn:mock-mu}), but,
as pointed out in the proof of
\cite{sally-taibleson:special}*{Theorem 3.1},
we have by Lemma 3.1 \loccit and \eqref{eq:sgn-ram} that it
simplifies to
\begin{align*}
\Gamma(\nu^{1/2}\sgn_\varpi)
& {}= \int_{\ord(x) = -1}
	\ol\Phi_\varpi(x)\abs x^{1/2}\sgn_\varpi(x)
\textup d\mult x \\
& {}= \int_{\pint\mult}
	\ol\Phi_\varpi(\varpi\inv x)
	\abs{\varpi\inv x}^{1/2}
	\sgn_\varpi(\varpi\inv x)
\textup d\mult x \\
& {}= q^{1/2}\sgn_\varpi(-1)
\meas_{\textup d\mult x}(1 + \pp)
\sum_{x \in \pint\mult/1 + \pp}
	\ol\Phi(x)\sgn_\resfld(x),
\end{align*}
where \(\textup d\mult x\) is the Haar measure on
\(\field\mult\) with respect to which
\(\meas_{\textup d\mult x}(\pint\mult)
= 1 - q\inv\)
(see Definition \ref{defn:k-Haar}).
Since
\(\meas_{\textup d\mult x}(1 + \pp) = q\inv\),
the result now follows from Lemma \ref{lem:G-facts}.
\end{proof}

We will also need some constants associated to specific
elements.

In \cite{waldspurger:loc-trace-form}*{Proposition VIII.1},
Waldspurger describes the `behaviour at \(\infty\)' of
Fourier transforms of semisimple orbital integrals on
general reductive, \(p\)-adic Lie algebras.
His description involves a \(4\)th root of unity
\(\gamma_\psi(X^*, Y)\) (cf.\ p.~79 \loccit);
since his \(\psi\) is our \(\Phi\)
(see Notation \ref{notn:Phi}), we denote it by
\(\gamma_\Phi(X^*, Y)\).
See Theorem \ref{thm:uniform} for our quantitative
analogues (for the special case of \(\sl_2\))
of his result.

Although we would like to do so
(see Remark \ref{rem:what-about}),
it is notationally unwieldy to avoid any longer choosing
`standard' representatives for \(\field\mult/(\field\mult)^2\).
Although our proofs will make use of these choices, none
of the statements of the main results
(except
Theorems \ref{thm:other-bad-ram} and \ref{thm:that-bad-ram},
via Remark \ref{rem:named-torus})
rely on them.

\begin{notn}
\label{notn:epsilon-and-varpi}
Let \(\epsilon\) be a lift to \(\pint\mult\) of a non-square
in \(\resfld\mult\),
and \(\varpi\) a uniformiser of \(\field\).
\end{notn}

\begin{defn}
\label{defn:Wald-i}
Recall Notations \ref{notn:X*} and \ref{notn:Y}.
If \(X^*\) and \(Y\) lie in stably conjugate tori,
so that \(\theta \equiv \theta' \pmod{(\field\mult)^2}\),
then
\[
\gamma_\Phi(X^*, Y)
= \begin{cases}
1, &
	\theta \equiv 1 \\
\gammaun(s), &
	\theta \equiv \epsilon \\
\gammaram(s), &
	\theta \equiv \varpi \\
-\gammaun(s)\gammaram(s), &
	\theta \equiv \epsilon\varpi,
\end{cases}
\]
where all congruences are taken modulo \((\field\mult)^2\),
and where
\[
\gammaun(s) \ldef (-1)^{\scdepth + 1}\sgn_\epsilon(s)
\qandq
\gammaram(s) \ldef \sgn_\varpi(-s)G_\varpi(\scPhi)
\]
(with notation as in
Notation \ref{notn:depth-and-scdepth}
and
Definition \ref{defn:G}).
It simplifies our notation considerably also to put
\(\gamma_\Phi(X^*, Y) = 1\)
if \(X^*\) is elliptic and \(Y\) is split,
and otherwise put
\(\gamma_\Phi(X^*, Y) = 0\)
if \(X^*\) and \(Y\) do not lie in stably conjugate tori.
\end{defn}

\begin{rem}
The dependence of \(\gamma_\Phi(X^*, Y)\) on \(X^*\) is
via \(\scdepth\) and \(\scPhi\)
(see Notation \ref{notn:depth-and-scdepth}).
Expanding these definitions shows that
\(\gamma_\Phi(X^*, Y)
= c_{\theta, \phi}\dotm\sgn_\theta(\beta s)\)
when \(X^*\) and \(Y\) lie in stably conjugate tori,
where the notation is as in
Notations \ref{notn:X*} and \ref{notn:Y}.

Notice that we have defined \(\gamma_\Phi(X^*, Y)\) only
when \(X^*\) and \(Y\) belong to (possibly different)
standard tori, in the sense of
Definition \ref{defn:tori}.
A direct computation shows that, if we replace \(X^*\) or
\(Y\) by a rational conjugate, or replace the pair
\((X^*, Y)\) by a stable conjugate,
such that \(X^*\) and \(Y\) still lie in standard tori,
then the constant \(\gamma_\Phi(X^*, Y)\) does not change.
(In the notation of Definition \ref{defn:sl2-as-fields},
\(\Ad^*(g)X^*\) lies in a standard torus if and only if
\(\varphi_\theta(g) = (\alpha, 0)\), in which case
\(\Ad^*(g)X^*
= \beta\Norm_\theta(\alpha)\dotm
	\sqrt{\Norm_\theta(\alpha)^{-2}\theta}\);
and similarly for \(Y\).)
This allows us to define \(\gamma_\Phi(X^*, Y)\)
for all pairs of regular, semisimple elements, if desired.
\end{rem}

By Lemma \ref{lem:G-facts},
\begin{equation}
\label{eq:gamma-square}
\gammaram(s)^2 = \sgn_\varpi(-1).
\end{equation}
In order to make use of
Propositions \ref{prop:Bessel-shallow}
and 
\ref{prop:Bessel-Kloost}
below, we will need the computation
\begin{multline}
\label{eq:sgn-and-G}
\sgn_\varpi(v)G_\varpi(\scPhi_{\varpi^{\scdepth + 1}}) \\
= \sgn_\varpi(\varpi^{-(\scdepth + 1)}s\theta)\dotm
\sgn_\varpi(\varpi^{\scdepth + 1})G_\varpi(\scPhi)
= \sgn_\varpi(-\theta)\gamma\textsub{ram}(s).
\end{multline}

\begin{rem}
\label{rem:what-about}
We will be interested exclusively in the case when
\(\theta\in \sset{1, \epsilon, \varpi}\).
This means that we seem to be omitting the cases when
\(\theta
\in \sset{\varpi^2\epsilon, \epsilon^2\varpi,
\epsilon^{\pm1}\varpi}\); but, actually, this problem is not
serious.
Indeed, for \(b \in \field\), write
\(g_b \ldef \begin{smallpmatrix}
1 & 0 \\
0 & b
\end{smallpmatrix} \in \GL_2(\field).\)
Then
\[
\Ad^*(g_b)X^* = \Ad^*(g_b)(\beta\dotm\sqrt\theta)
= \beta b\inv\dotm\sqrt{b^2\theta}
\]
(where we identify \(\ttt_\theta^*\) with \(\ttt_\theta\) via the
trace pairing, as in Notation \ref{notn:X*}); and
\(\hat\mu^G_{X^*}
= \hat\mu^G_{\Ad^*(g_b)X^*} \circ \Ad(g_b)\).
This covers
\(\theta = \varpi^2\epsilon\) (by taking \(b = \varpi\inv\))
and
\(\theta = \epsilon^2\varpi\) (by taking \(b = \epsilon\inv\)).
Handling \(\theta \in \sset{\epsilon^{\pm1}\varpi}\)
requires a different observation:
since our
choice of uniformiser was arbitrary, it could as well have
been \(\epsilon^{\pm1}\varpi\)
(or, for that matter, \(\epsilon^2\varpi\))
as \(\varpi\) itself.
Thus, the formul\ae\ for the cases
\(\theta = \epsilon^n\varpi\)
can be obtained by
simple substitution.

The definition of \(\gamma_\Phi(X^*, Y)\) when
\(\theta \equiv \epsilon\varpi \pmod{(\field\mult)^2}\) is an
instance of this;
namely, by Lemma \ref{lem:G-facts},
\begin{align*}
-\gammaun(s)\gammaram(s)
& {}= (-1)^\scdepth\sgn_\epsilon(s)\dotm
	\sgn_\varpi(-s)G_\varpi(\scPhi) \\
& {}= \sgn_{\epsilon\varpi}(-s)\dotm
	\sgn_\varpi(\epsilon)^\scdepth G_\varpi(\scPhi) \\
& {}= \sgn_{\epsilon\varpi}(-s)G_{\epsilon\varpi}(\scPhi),
\end{align*}
where we have used that \(\sgn_\epsilon(-1) = 1\) and
\(\sgn_\varpi(\epsilon) = -1\).
\end{rem}

We next define a constant \(c_0(X^*)\) for use in
Theorem \ref{thm:close-spun} and \ref{thm:close-ram}.
Those theorems
(and Proposition \ref{prop:orbital-integral-integral})
will show that, as the notation suggests,
it is the coefficient of the trivial orbit
in the expansion of the germ of \(\hat\mu^G_{X^*}\)
in terms of Fourier transforms of nilpotent orbital integrals
(see \cite{hc:queens}*{Theorem 5.11}).

\begin{defn}
\label{defn:const}
\[
c_0(X^*) = \begin{cases}
-2q\inv, &
	\text{\(X^*\) split} \\
-q\inv,  &
	\text{\(X^*\) unramified} \\
-\tfrac1 2 q^{-2}(q + 1), &
	\text{\(X^*\) ramified.}
\end{cases}
\]
\end{defn}

Recall that \(\hat\mu^G_{X^*}\) is defined in terms of the
measure \(\textup d\dot g\) of
Proposition \ref{prop:orbital-integral-integral};
and note that, in the notation of that proposition,
\[
c_0(X^*) = (q - 1)\inv\meas_{\textup d\dot g}(\dot K)
\]
whenever \(X^*\) is elliptic.

\section{Bessel functions}
\label{sec:Bessel}

Our strategy for computing Fourier transforms of orbital
integrals is to reduce them to \(p\)-adic Bessel functions
(see Proposition \ref{prop:second-orbital},
\eqref{eq:third-orbital-un},
and
\eqref{eq:third-orbital-ram}).
In this context, we are referring to the complex-valued
Bessel functions defined in
\cite{sally-taibleson:special}*{\S4},
not the \(p\)-adic-valued ones defined in
\cite{dwork:bessel}.
%

The definition of these functions
depends on an additive character,
denoted by \(\chi\) in \cite{sally-taibleson:special},
and a multiplicative character,
there denoted by \(\pi\),
of \(\field\).
For internal consistency, we will instead denote the additive
character by \(\Phi\) and the multiplicative character by
\(\chi\);
but, for consistency with their work, we shall require
throughout this section that \(\depth(\Phi) = -1\), i.e.,
that \(\Phi\) is trivial on \(\pint\) but not on
\(\pp\inv\).

\begin{defn}[\cite{sally-taibleson:special}*{(4.1)}]
\label{defn:Bessel}
For \(\chi \in \widehat{\field\mult}\), the
\term{\(p\)-adic Bessel function of order \(\chi\)} is given
by
\[
\indexmem{\Bessel_\chi}(u, v)
= \Pint_{\field\mult}
	\Phi(u x + v x\inv)\chi(x)
\textup d\mult x
\quad\text{for \(u, v \in \field\mult\),}
\]
where \(\textup d\mult x\) is the Haar measure on
\(\field\mult\) fixed in Definition \ref{defn:k-Haar}.
We also put
\(\indexmem{\Bessel_\chi^\theta}
= \tfrac1 2(\Bessel_\chi + \Bessel_{\chi\sgn_\theta})\),
with notation as in Definition \ref{defn:maps-and-gps}.
\end{defn}

The locally constant \(K\)-Bessel function
\(K(z \mid \chi)\)
of \cite{trimble:p-adic-bessel}*{Definition 3.2}
is \(\Bessel_\chi(\varpi^t, \varpi^t)\)
(in the notation of that definition),
where \(\varpi\) is a uniformiser.

Note that, for \(\chi \ne 1\), it is natural to extend the
Bessel function by putting
\(\Bessel_\chi(u, 0) = \chi(u)\inv\Gamma(\chi)\)
and
\(\Bessel_\chi(0, v) = \chi(v)\Gamma(\chi\inv)\),
where the \(\Gamma\)-factors are as in
\cite{sally-taibleson:special}*{\S3},
and that, under some conditions on \(\chi\), we can even
define \(\Bessel_\chi(0, 0)\) (either as \(0\) or the sum of
a geometric series);
but we do not need to do this.

The notation \(\Bessel_\chi^\theta\) arises naturally in our
computations; see Proposition \ref{prop:second-orbital}.

\begin{defn}
We say that a character \(\chi \in \widehat{\field\mult}\)
is \term{mildly ramified} if \(\chi\) is trivial on
\(1 + \pp\), but non-trivial on \(\field\mult\).
\end{defn}

Since our orbital-integral calculations require
information about \(\Bessel_\chi\) only for \(\chi\) mildly
ramified, and since more precise information is available in
that case in general, it is there that we focus our
attention.

\begin{notn}
\label{notn:generic-u-v-and-m}
We fix the following notation for the remainder of the
section.
\begin{itemize}
\item
\(u, v \in \field\mult\);
\item
\(m = -\ord(u v)\);
and
\item
\(\chi \in \widehat{\field\mult}\).
\end{itemize}
\end{notn}

This is consistent with Notation \ref{notn:u-v-and-m}.
After Proposition \ref{prop:Bessel-shallow},
we will assume that \(\chi\) is mildly ramified.

Of particular interest to us later will be the cases
where \(\chi\) is an unramified twist of
one of the characters \(\sgn_{\theta'}\) of Definition
\ref{defn:maps-and-gps} (i.e., is of the form
\(\nu^\alpha\sgn_\theta\) for some \(\alpha \in \C\)).
Note that \(\sgn_\epsilon = \nu^{\pi i/\ln(q)}\).

\begin{thm}[%
	Theorems 4.8 and 4.9 of \cite{sally-taibleson:special}%
]
\label{thm:sally-taibleson:special:bessel}
\[
J_\chi(u, v) = \begin{cases}
\chi(v)\Gamma(\chi\inv) + \chi(u)\inv\Gamma(\chi), &
	m \le 1 \\
\chi(u)\inv F_\chi(m/2, u v), &
	\text{\(m \ge 2\) and \(m\) even} \\
0, &
	\text{\(m > 2\) and \(m\) odd,}
\end{cases}
\]
where
the \(\Gamma\)-factors are as in
\cite{sally-taibleson:special}*{\S3},
and
\[
F_\chi(m/2, u v)
\ldef \int_{\ord(x) = -m/2}
	\Phi(x + u v x\inv)\chi(x)
\textup d\mult x.
\]
\end{thm}

The \(\Gamma\)-factor tables of
\cite{sally-taibleson:special}*{Theorem 3.1},
together with Lemma \ref{lem:Gamma-varpi},
mean that we understand \(\Bessel_\chi(u, v)\) completely
when \(m < 2\), but further calculation is necessary in the
remaining cases.

\begin{prop}
\label{prop:Bessel-shallow}
If
\begin{itemize}
\item \(h \in \Z_{> 0}\),
\item \(\chi\) is trivial on \(1 + \pp^h\),
and
\item \(m \ge 4h - 1\),
\end{itemize}
then
\(\Bessel_\chi(u v) = 0\)
if \(u v \not\in (\field\mult)^2\);
and, if \(w \in \field\mult\) satisfies \(u v = w^2\), then
\begin{align*}
\Bessel_\chi(u, v) ={} & q^{-m/4}\chi(u\inv w) \times{} \\
& \begin{cases}
\Phi(2w) + \chi(-1)\Phi(-2w), & 4 \mid m \\
\sgn_\varpi(w)G_\varpi(\Phi)\bigl(
	\Phi(2w) + (\chi\sgn_\varpi)(-1)\Phi(-2w)
\bigr), & 4 \nmid m \\
\end{cases}
\end{align*}
\end{prop}
%

\begin{proof}
If \(m\) is odd, then the vanishing result follows from
Theorem \ref{thm:sally-taibleson:special:bessel}, so we assume
that \(m\) is even.
In this case,
\(m \ge 4h\); and,
by Theorem \ref{thm:sally-taibleson:special:bessel},
\(\Bessel_\chi(u, v) = \chi(u)\inv F_\chi(m/2, u v)\).

We evaluate the integral defining \(F_\chi(m/2, u v)\)
by splitting it into pieces.
Write
\begin{align*}
S_{u v} & {}= \sett{x \in \field}{%
	\(\ord(x) = -m/2\) and \(\ord(x - u v x\inv) < -m/2 + h\)%
} \\
\intertext{and}
T_{u v} & {}= \sett{x \in \field}{%
	\(\ord(x) = -m/2\) and \(\ord(x - u v x\inv) \ge -m/2 + h\)%
}.
\end{align*}
Note that both \(S_{u v}\) and \(T_{u v}\) are invariant
under multiplication by \(1 + \pp\);
and that, if \(x \in T_{u v}\), then
\(u v \in x^2(1 + \pp^h) \subseteq (\field\mult)^2\).
We claim that the relevant integral may be taken over only
\(T_{u v}\).

If \(X \in \pp^{m/2 - h}\),
then we have by Lemma \ref{lem:cayley-facts}
and the fact that \(2(m/2 - h) \ge m/2\) that
\[
\cayley(X)     {}\equiv 1 + 2X \pmod{\pp^{m/2}}
\qandq
\cayley(X)\inv {}\equiv 1 - 2X \pmod{\pp^{m/2}},
\]
so
\begin{align*}
&\int_{S_{u v}}
	\Phi(x + u v x\inv)\chi(x)
\textup d\mult x \\
& \qquad= (\star)\int_{\pp^{m/2 - h}}
	\int_{S_{u v}}
		\Phi\bigl(
			x\dotm\cayley(X) + u v x\inv\dotm\cayley(X)\inv
		\bigr)
		\chi\bigl(x\dotm\cayley(X)\bigr)
	\textup d\mult x\,
\textup dX \\
& \qquad= (\star)\int_{S_{u v}}
	\Phi(x + u v x\inv)\chi(x)
	\int_{\pp^{m/2 - h}}
		\Phi_{2(x - u v x\inv)}(X)
	\textup dX\,
\textup d\mult x,
\end{align*}
where \((\star) = \meas_{\textup dX}(\pp^{m/2 - h})\inv\)
is a constant.
We used that \(\Phi\) is trivial on
\(x\pp^{m/2} \cup u v x\inv\pp^{m/2}
	\subseteq \pint\)
and
\(\chi\) is trivial on
\(\cayley(\pp^{m/2 - h}) = 1 + \pp^{m/2 - h} \subseteq 1 + \pp^h\).
By \eqref{eq:depth-Phi-b},
we have that
\(\depth(\Phi_{2(x - u v x\inv)}) > m/2 - h + 1\)
(i.e., \(\Phi_{2(x - u v x\inv)}\) is a
non-trivial character on \(\pp^{m/2 - h}\))
whenever \(x \in S_{u v}\), so the inner integral
is \(0\).
This shows that, as desired, the integral defining
\(F_\chi(m/2, u v)\) may be taken over only \(T_{u v}\).

If \(u v \not\in (\field\mult)^2\), then
\(T_{u v} = \emptyset\), so
\(\Bessel_\chi(u, v) = \chi(u)\inv F_\chi(m/2, u v) = 0\);
whereas, if \(w \in \field\mult\) satisfies \(w^2 = u v\),
then \(T_{u v} = w(1 + \pp^h) \sqcup -w(1 + \pp^h)\), so
\begin{equation}
\tag{$*$}
\label{eq:slim-Bessel}
\Bessel_\chi(u, v)
= \chi(u)\inv\Bigl(
	\int_{w(1 + \pp^h)}
		\Phi(x + u v x\inv)\chi(x)
	\textup d\mult x +
	\int_{-w(1 + \pp^h)}
		\Phi(x + u v x\inv)\chi(x)
	\textup d\mult x
\Bigr).
\end{equation}
Note that \(\ord(w) = -m/2\).

We show a detailed calculation of the first integral; of
course, that of the second is identical.
Note that the integral no longer involves \(\chi\).
By Lemma \ref{lem:cayley-facts} again,
we have that \(X \mapsto w\dotm\cayley(X)\)
is a measure-preserving bijection
from \(\pp^h\) to \(w(1 + \pp^h)\),
so
\[
\int_{w(1 + \pp^h)}
	\Phi(x + u v x\inv)\chi(x)
\textup d\mult x
= \chi(w)\int_{\pp^h}
	\Phi_w\bigl(\cayley(X) + \cayley(X)\inv\bigr)
\textup dX,
\]
where we have used that \(u v w\inv = w\)
and again that \(\chi\) is trivial on
\(\cayley(\pp^h) = 1 + \pp^h\).
We will evaluate the latter integral by breaking it into
`shells' on which \(\ord(X)\) is constant,
using the following facts.
Note that, by direct computation
(and Definition \ref{defn:cayley}),
\[
\cayley(X) + \cayley(X)\inv = 2\cayley(X^2)
\]
for \(X \in \field \setminus \sset1\).
If \(\ord(X) = i\) and \(\ord(Y) = j\), then
we have by Lemma \ref{lem:cayley-facts} once more that
\begin{align*}
\cayley\bigl((X + Y)^2\bigr) &
	{}\equiv \cayley(X^2 + 2X Y) \pmod{\pp^{2j}} \\
\intertext{and}
\cayley(X^2 + 2X Y)          &
	{}\equiv \cayley(X^2) + 4X Y \pmod{\pp^{2j}}.
\end{align*}
(In fact, the second congruence could be made much finer,
but that would be of no use here.)

In particular, fix \(i \ge h\) with \(2i < m/2 - 1\),
so that
\(\depth(\Phi) = m/2 - 1 < 2(m/2 - 1 - i)\)
(i.e., \(\Phi\) is trivial on \(\pp^{2(m/2 - 1 - i)}\)).
Then
\begin{align*}
&\int_{\ord(X) = i}
	\Phi_w\bigl(\cayley(X) + \cayley(X)\inv\bigr)
\textup dX \\
& \qquad= (\star)\int_{\pp^{m/2 - 1 - i}}
	\int_{\ord(X) = i}
		(\Phi_{2w} \circ \cayley)\bigl((X + Y)^2\bigr)
	\textup dX\,
\textup dY \\
& \qquad= (\star)\int_{\ord(X) = i}
	(\Phi_{2w} \circ \cayley)(X^2)
	\int_{\pp^{m/2 - 1 - i}}
		\Phi_{8w X}(Y)
	\textup dY\,
\textup dX,
\end{align*}
where \((\star) = \meas(\pp^{m/2 - 1 - i})\inv\) is a
constant.
Since
\(\depth(\Phi_{8w X}) = \depth(\Phi_w) - \ord(8X)
	= m/2 - 1 - i\),
the inner integral is \(0\).

Note that \(\rup{(m/2 - 1)/2} \ge h\).  We have thus shown
that
\[
\Bessel_\chi(u, v)
= \int_{\pp^{\rup{(m/2 - 1)/2}}}
	(\Phi_{2w} \circ \cayley)(X^2)
\textup dX.
\]
If \(m/2\) is even, then the integral is over \(\pp^{m/4}\),
and
\(\cayley(X^2) \equiv 1
	\pmod{\pp^{m/2} \subseteq \ker \Phi_{2w}}\)
for all \(X \in \pp^{m/4}\).  Thus, in that case,
\[
\Bessel_\chi(u, v)
= \meas_{\textup dX}(\pp^{m/4})\Phi_{2w}(1)
= q^{-m/4}\Phi(2w).
\]
If \(m/2\) is odd, then the integral is over
\(\pp^{m/4 - 1/2}\),
and
\(\cayley(X^2) \equiv 1 + 2X^2 \pmod{\pp^{m/2}}\)
for all \(X \in \pp^{m/4 - 1/2}\).  Thus, in that case,
\begin{align*}
\Bessel_\chi(u, v)
& {}= \meas_{\textup dX}(\pp^{m/4 + 1/2})\Phi_{2w}(1)
\sum_{X \in \pp^{m/4 - 1/2}/\pp^{m/4 + 1/2}}
	\Phi_{4w}(X^2) \\
& {}= q^{-m/4}\Phi(2w)q^{-1/2}
\sum_{X \in \pint/\pp}
	\Phi_{4w\varpi^{m/2 - 1}}(X^2).
\end{align*}
By Lemma \ref{lem:G-facts},
and the fact that \(m/2\) is odd,
this can be re-written as
\[
q^{-m/4}\Phi(2w)\sgn_\varpi(-1)^{m/2 - 1}G_\varpi(\Phi_{4w})
= q^{-m/4}\Phi(2w)\sgn_\varpi(w)G_\varpi(\Phi).
\]
The result now follows from \eqref{eq:slim-Bessel}.
\end{proof}

From now on, we assume that \(\chi\) is mildly ramified.
In particular,
we may take \(h = 1\), so that
Proposition \ref{prop:Bessel-shallow} holds
whenever \(m > 2\).

\begin{defn}
For
\begin{itemize}
\item
\(\xi \in \resfld\mult\),
\item
\(\ol\Phi\) an (additive) character of \resfld,
and
\item
\(\ol\chi\) a (multiplicative) character of
\(\resfld\mult\),
\end{itemize}
we define the corresponding
\term{twisted Kloosterman sum} by
\[
\indexmem{\Kloost(\ol\chi, \ol\Phi; \xi)}
\ldef \sum_{x \in \resfld\mult}
	\ol\Phi(x + \xi x\inv)\ol\chi(x).
\]
\end{defn}

\begin{prop}
\label{prop:Bessel-Kloost}
If \(m = 2\), then
\[
\Bessel_\chi(u, v)
= q\inv\chi(u\varpi)\inv\Kloost(\ol\chi, \ol\Phi; \xi),
\]
Here,
\begin{itemize}
\item
\(\xi\) is the image in \(\resfld\mult\) of
\(\varpi^2 u v \in \pint\mult\),
\item
\(\ol\Phi\) is the (additive) character of
\(\resfld = \pint/\pp\) arising from the restriction to
\(\pint\) of the depth-\(0\) character \(\Phi_{\varpi\inv}\)
of \(\field\),
and
\item
\(\ol\chi\) is the (multiplicative) character of
\(\resfld\mult \cong \pint\mult/1 + \pp\)
arising from the restriction to \(\pint\mult\) of \(\chi\).
\end{itemize}
\end{prop}

\begin{proof}
By Theorem \ref{thm:sally-taibleson:special:bessel},
\begin{align*}
\chi(u\varpi)\Bessel_\chi(u, v)
& {}= \chi(\varpi)\int_{\ord(x) = -1}
	\Phi(x + u v x\inv)\chi(x)
\textup d\mult x \\
& {}= \int_{\pint\mult}
	\Phi(\varpi\inv x + u v\dotm\varpi x\inv)
	\chi(x)
\textup d\mult x \\
& {}= \meas_{\textup d\mult x}(1 + \pp)
\sum_{x \in \pint\mult/1 + \pp}
	\Phi_{\varpi\inv}(x + \varpi^2 u v x\inv)
	\chi(x)
\textup d\mult x.
\end{align*}
Since \(\meas_{\textup d\mult x}(1 + \pp) = q\inv\), the result follows.
\end{proof}

\begin{cor}
\label{cor:Bessel-Kloost}
Suppose that \(m = 2\).
Then
\begin{align*}
\Bessel_{\nu^\alpha}(u, v)
& {}= q^{\alpha - 1}\abs u^{-\alpha}
\sum_{\substack{
	c \in \pp\inv/\pint \\
	c^2 \ne u v
}}
	\Phi(2c)\sgn_\varpi(c^2 - u v) \\
\Bessel_{\nu^\alpha\sgn_\varpi}(u, v)
& {}= q^{\alpha - 1/2}\abs u^{-\alpha}
\sgn_\varpi(v)G_\varpi(\Phi)
\sum_{\substack{
	c \in \pp\inv/\pint \\
	c^2 = u v
}}
	\Phi(2c)
\end{align*}
for \(\alpha \in \C\).
\end{cor}

\begin{proof}
If \(\chi = \nu^\alpha\), then
\(\ol\chi = 1\), so
\cite{lidl-niederreiter:finite-fields}*{Theorem 5.47}
gives that
\begin{align*}
\Kloost(\ol\chi, \ol\Phi; \xi)
& {}= \sum_{\substack{
	c \in \resfld   \\
	c^2 \ne \xi
}}
	\ol\Phi(2c)\sgn_\resfld(c^2 - \alpha) \\
& {}= \sum_{\substack{
	c \in \pint/\pp \\
	c^2 \ne \varpi^2 u v
}}
	\ol\Phi(2c)
	\sgn_\varpi(c^2 - \varpi^2 u v) \\
& {}= \sum_{\substack{
	c \in \pp\inv/\pint \\
	c^2 \ne u v
}}
	\Phi(2c)
	\sgn_\varpi(c^2 - u v).
\end{align*}
(Note that our \(\ol\Phi\) is their \(\chi\), and that they
write \(K(\chi; a, b)\) where we write
\(\Kloost(\ol\Phi, 1; a b)\).)

If \(\chi = \nu^\alpha\sgn_\varpi\), then
\(\ol\chi = \sgn_\resfld\), so
\cite{lidl-niederreiter:finite-fields}*{%
	Exercises 5.84--85%
}
gives that
\begin{align*}
\Kloost(\ol\chi, \ol\Phi; \xi)
& {}= \sgn_\resfld(\xi)G(\sgn_\resfld, \ol\Phi)
\sum_{\substack{
	c \in \resfld \\
	c^2 = \xi
}}
	\ol\Phi(2c) \\
& {}= \sgn_\varpi(u v)G(\sgn_\resfld, \ol\Phi)
\sum_{\substack{
	c \in \pp\inv/\pint \\
	c^2 = u v
}}
	\Phi(2c),
\end{align*}
where
\(G(\sgn_\resfld, \ol\Phi)
= \sum_{X \in \resfld\mult} \ol\Phi(X)\sgn_\resfld(X)\).
(Note that our \(\ol\Phi\) is their \(\chi\)
and our \(\ol\chi\) their \(\eta\),
and that
they write \(K(\eta, \chi; 1, \xi)\) where we write
\(K(\ol\chi, \ol\Phi; \xi)\).)
Since \(\depth(\Phi) = -1\),
Lemma \ref{lem:G-facts} gives that
\(G(\sgn_\resfld, \ol\Phi)
= q^{1/2}\sgn_\varpi(-1)G_\varpi(\Phi)\).

The result now follows from Proposition \ref{prop:Bessel-Kloost}.
\end{proof}

We now state an apparently rather specialised corollary,
which nonetheless turns out to be sufficient to simplify
many of our `shallow' computations (see
\S\ref{sec:shallow-spun} and \S\ref{sec:shallow-ram}).

\begin{cor}
\label{cor:Bessel-twist}
If \(m \ge 2\) and \(\ord(u) = \ord(v)\),
then
\(\Bessel_{\nu^\alpha\chi}(u, v)\) is independent of
\(\alpha \in \C\);
in particular,
\[
\Bessel_\chi^\epsilon(u, v) = \Bessel_\chi(u, v)
\qandq
\Bessel_\chi^\varpi(u, v)
= \Bessel_{\chi\sgn_\epsilon}^\varpi(u, v).
\]
If \(m \ge 2\) and \(\ord(u) = \ord(v) + 2\), then
\(\Bessel_{\nu^\alpha\chi}(u, v)
= q^\alpha\Bessel_\chi(u, v)\);
in particular,
\[
\Bessel_\chi^\epsilon(u, v) = 0
\qandq
\Bessel_\chi^\varpi(u, v)
= -\Bessel_{\chi\sgn_\epsilon}^\varpi(u, v).
\]
\end{cor}

\begin{proof}
Suppose that \(m > 2\).
If \(u v \not\in (\field\mult)^2\), then
\(\Bessel_{\nu^\alpha\chi}(u, v) = 0\)
for all \(\alpha \in \C\).
If \(u v = w^2\), then the only dependence on \(\alpha\) in
Proposition \ref{prop:Bessel-shallow} is via the factor
\(\chi(u\inv w)\).
If \(\ord(u) = \ord(v)\), then also \(\ord(w) = \ord(u)\),
so \(\nu^\alpha(u\inv w) = 1\).
If \(\ord(u) = \ord(v) + 2\), then \(\ord(w) = \ord(u) - 1\),
so \(\nu^\alpha(u\inv w) = q^\alpha\).

Now suppose that \(m = 2\), i.e., that \(\ord(u v) = -2\).
Since \(\ol{\nu^\alpha\chi} = \ol\chi\), the only dependence
on \(\alpha\) in Proposition \ref{prop:Bessel-Kloost} is via the factor
\(\chi(u\varpi)\inv\).  If
\(\ord(u) = \ord(v)\), then \(\ord(u) = -1\), so
\(\nu^\alpha(u\varpi) = 1\).
If \(\ord(u) = \ord(v) + 2\), then \(\ord(u) = 0\), so
\(\nu^\alpha(u\varpi) = q^{-\alpha}\).
\end{proof}

\section{A mock-Fourier transform}
\label{sec:mock-mu}

We begin by introducing a function \(M^G_{X^*}\) specified
by an integral formula (see Definition \ref{defn:mock-mu})
reminiscent of the usual one
for (the function representing) \(\hat\mu^G_{X^*}\)
(see \cite{adler-debacker:mk-theory}*{Theorem A.1.2}).
We will eventually show
(see Proposition \ref{prop:orbital-integral-integral}) that
it is actually \emph{equal} to \(\hat\mu^G_{X^*}\), but
first we spend some time computing it.

In the notation of Definition \ref{defn:tori},
we have
\begin{equation}
\label{eq:funny-trace}
\Tr g\dotm\sqrt\theta\dotm g\inv\dotm\sqrt{\theta'}
= \Norm_\theta(\alpha)\dotm\theta' + \Norm_\theta(\gamma),
\end{equation}
where \(g = \begin{smallpmatrix}
a & b \\
c & d
\end{smallpmatrix}\), \(\alpha = a + b\sqrt\theta\), and
\(\gamma = c + d\sqrt\theta\).
Since \(1 = a d - b c = \Im_\theta(\ol\alpha\dotm\gamma)\),
we have that
\(\gamma = \smash{\ol\alpha}\inv\dotm(t + \sqrt\theta)\)
for some \(t \in \field\); specifically,
\(t = \Re_\theta(\ol\alpha\dotm\gamma) = a c - b d\theta\).
This calculation motivates the definition of the following
map.

\begin{defn}
\label{defn:sl2-as-fields}
We define
\(\varphi_\theta : G \to \field_\theta\mult \times \field\)
by
\[
\varphi_\theta(g) = (a + b\sqrt\theta, a c - b d\theta)
\]
for \(g = \begin{smallpmatrix}
a & b \\
c & d
\end{smallpmatrix} \in G\).
\end{defn}

Note that \(\varphi_\theta\) is a bi-analytic map
(of \(\field\)-manifolds), with inverse
\[
(\alpha, t) \mapsto \begin{pmatrix}
\Re_\theta(\alpha) &
	\Im_\theta(\alpha) \\
\Norm_\theta(\alpha)\inv\bigl(
	t\dotm\Re_\theta(\alpha) + \theta\dotm\Im_\theta(\alpha)
\bigr) &
\Norm_\theta(\alpha)\inv\bigl(
	\Re_\theta(\alpha) + t\dotm\Im_\theta(\alpha)
\bigr)
\end{pmatrix}.
\]
It is \emph{not} an isomorphism, but its restrictions to
\(T_\theta\),
\(A\),
and
\(\set{\begin{smallpmatrix}
1 & 0 \\
b & 1
\end{smallpmatrix}}{b \in \field}\)
are isomorphisms onto
\(C_\theta \times \sset0\),
\(\field\mult \times \sset0\),
and
\(\sset1 \times \field\),
respectively.  In fact, the next lemma says a bit more.

\begin{lemma}
\label{lem:torus-acts}
If \(g \in G\) satisfies \(\varphi(g) = (\alpha, t)\),
and
\begin{itemize}
\item
\(h \in T_\theta\) is identified with \(\eta \in C_\theta\),
\item
\(a = \begin{smallpmatrix}
\lambda & 0           \\
0       & \lambda\inv
\end{smallpmatrix}\)
(with \(\lambda \in \field\mult\)),
and
\item
\(\ol u = \begin{smallpmatrix}
1 & 0 \\
b & 1
\end{smallpmatrix}\)
(with \(b \in \field\)),
\end{itemize}
then
\begin{align*}
\varphi_\theta(g h)     & {}= (\alpha\eta, t),                     \\
\varphi_\theta(a g)     & {}= (\lambda\alpha, t),                  \\
\intertext{and}
\varphi_\theta(\ol u g) & {}= (\alpha, t + \Norm_\theta(\alpha)b).
\end{align*}
\end{lemma}

\begin{proof}
This is a straightforward computation.
\end{proof}

Now we are in a position to define our
`mock orbital integral'.
Again, Proposition \ref{prop:orbital-integral-integral}
will eventually show that it is actually equal to
the function in which we are interested.

\begin{defn}
\label{defn:mock-mu}
For \(\alpha \in \field_\theta\mult\) and \(t \in \field\),
put
\[
\indexmem{\langle X^*, Y\rangle_{\alpha, t}}
\ldef \beta s\bigl(
	\Norm_\theta(\alpha)\dotm\theta' +
	\Norm_\theta(\alpha)\inv\dotm\theta -
	\Norm_\theta(\alpha)\inv\dotm t^2
\bigr).
\]
Notice that the dependence on \(\alpha\) is
only via \(\Norm_\theta(\alpha)\).  Thus, we may
define
\[
\indexmem{M^G_{X^*}(Y)}
\ldef \Pint_{\field_\theta\mult/C_\theta} \Pint_\field
	\Phi(\langle X^*, Y\rangle_{\alpha, t})
\textup dt\,\textup d\mult\dot\alpha,
\]
where
\begin{align*}
\Pint_\field f(x)\textup dt
& {}\ldef \sum_{n \in \Z}
	\int_{\ord(x) = n} f(x)\textup dt \\
\Pint_{\field\mult} f(x)\textup d\mult x
& {}\ldef \sum_{n \in \Z}
	\int_{\ord(x) = n} f(x)\textup d\mult x \\
\intertext{and}
\Pint_{\field_\theta\mult/C_\theta}
	(f \circ \Norm_\theta)(\alpha)
\textup d\mult\dot\alpha
& {}\ldef \Pint_{\field\mult}
	[\Norm_\theta(\field\mult)](x)f(x)
\textup d\mult x
\end{align*}
(for those \(f \in C^\infty(\field)\) for which the sum
converges)
are `principal-value' integrals, as in
\cite{sally-taibleson:special}*{p.~282}.
Here,
\(\textup dt\) and \(\textup d\mult x\)
are the measures of Definition \ref{defn:k-Haar},
and
\([S]\) denotes the characteristic function of \(S\).
\end{defn}

By \eqref{eq:funny-trace}
(and Notations \ref{notn:X*} and \ref{notn:Y}), we have that
\begin{equation}
\label{eq:hilarious-trace}
\langle X^*, Y\rangle_{\alpha, t}
= \langle\Ad^*(g)X^*, Y\rangle
\quad\text{when \(\varphi_\theta(g) = (\alpha, t)\),}
\end{equation}
where the pairing \(\langle\cdot, \cdot\rangle\) on the
right is the usual pairing between \(\gg^*\) and \(\gg\).

\begin{notn}
\label{notn:u-v-and-m}
\(\indexmem u = \varpi^{-(\scdepth + 1)}s\theta'\),
\(\indexmem v = \varpi^{-(\scdepth + 1)}s\theta\),
and
\(\indexmem m = -\ord(u v)\).
\end{notn}

This is a special case of
Notation \ref{notn:generic-u-v-and-m}.
These particular values of \(u\) and \(v\) will be fixed for
the remainder of the paper.
It follows that
\begin{equation}
\label{eq:uv}
u v = (\varpi^{-(\scdepth + 1)}s)^2\dotm\theta\theta',
\end{equation}
so
\begin{equation}
\label{eq:uv-is-square}
u v \in (\field\mult)^2 \Leftrightarrow
	\theta\theta' \in (\field\mult)^2;
\end{equation}
and we use Lemma \ref{lem:Y-facts} to compute
\begin{gather}
\label{eq:ord-u}
\ord(u) = -(\scdepth + 1) + \ord(s\theta')
= -\bigl(\scdepth + 1 + \tfrac1 2\ord(\theta')\bigr) +
	\depth(Y) \\
\intertext{and}
\label{eq:m-and-depth-Y}
m = 2(\scdepth + 1) - \ord(s^2\theta') - \ord(\theta)
= 2(\scdepth + 1 - \depth(Y)) - \ord(\theta).
\end{gather}

\subsection{Mock-Fourier transforms and Bessel functions}

We can now evaluate the integral occurring in
Definition \ref{defn:mock-mu} in terms of Bessel functions%
---or, rather, the sums \(\Bessel_\chi^\theta\) of
Definition \ref{defn:Bessel}.

\begin{prop}
\label{prop:second-orbital}
\begin{align*}
M^G_{X^*}(Y)
={} &
\tfrac1 2\abs s^{-1/2}q^{-(\scdepth + 1)/2}\times{} \\
& \quad\Bigl(
	\bigl(
		\Bessel_{\nu^{1/2}}^\theta(u, v) +
		\gammaun(s)
			\Bessel_{\nu^{1/2}\sgn_\epsilon}^\theta(u, v)
	\bigr) +{} \\
& \quad\qquad
	\gammaram(s)\bigl(
		\Bessel_{\nu^{1/2}\sgn_\varpi}^\theta(u, v) -
		\gammaun(s)
			\Bessel_{\nu^{1/2}\sgn_{\epsilon\varpi}}^\theta(u, v)
	\bigr)
\Bigr),
\end{align*}
where \(\Bessel_\chi^\theta\) is as in
Definition \ref{defn:Bessel}, and
\(\gammaun(s)\) and \(\gammaram(s)\) are as in
Definition \ref{defn:Wald-i}.
\end{prop}

\begin{proof}
Recall the notation \(\scPhi = \Phi_\beta\) from
Notation \ref{notn:depth-and-scdepth}.
By Definition \ref{defn:mock-mu},
\begin{equation}
\tag{$*$}
\label{eq:orbital-as-Pint}
\begin{aligned}
M^G_{X^*}&(Y) \\
& {}= \Pint_{\field_\theta\mult/C_\theta}
	\scPhi_s\bigl(
		\Norm_\theta(\alpha)\dotm\theta'
		+ \Norm_\theta(\alpha)\inv\dotm\theta
	\bigr)\dotm
	\Pint_\field
		\scPhi(-s \Norm_\theta(\alpha)\inv t^2)
	\textup dt\,
\textup d\mult\dot\alpha \\
& {}= q^{-(\scdepth + 1)/2}\Pint_{\field\mult}
	[\Norm_\theta(\field_\theta\mult)](x)
	\bessel(\theta', \theta; x)
	\mc H(\scPhi, -s x\inv)
\textup d\mult x,
\end{aligned}
\end{equation}
where
\begin{itemize}
\item
\(\bessel(\theta', \theta; x)
\ldef \scPhi_s(\theta'x + \theta x\inv)
= \Phi\bigl(\beta s(\theta'x + \theta x\inv)\bigr)\)
for \(x \in \field\mult\);
and
\item
\(\mc H(\scPhi, b)
= {\displaystyle\Pint_\field}
	\scPhi(b t^2)
\textup d_{\scPhi}t\)
for \(b \in \field\mult\)
is as in \cite{shalika:thesis}*{p.~6}.
\end{itemize}
In particular, \(\textup d_{\scPhi}t\) is the
\(\scPhi\)-self-dual Haar measure on \(\field\); by
\cite{shalika:thesis}*{p.~5}, it satisfies
\(\textup dt = q^{-(\scdepth + 1)/2}\textup d_{\scPhi}t\).
This is the reason for the appearance of
\(q^{-(\scdepth + 1)/2}\) on the last line of the
computation.

The significance of \(\bessel\) is that integrating it
against a (multiplicative) character \(\chi\) of
\(\field\mult\) corresponds to evaluating a Bessel function
of order \(\chi\), in the sense of Definition
\ref{defn:Bessel}.
To be precise, note that our character \(\scPhi\) has depth
\(\scdepth\), not \(-1\), so that we must work instead with
\(\scPhi_{\varpi^{\scdepth + 1}}\).
Then
\[
\bessel(\theta', \theta; x) =
\scPhi_{\varpi^{\scdepth + 1}}\bigl(
	(\varpi^{-(\scdepth + 1)}s\theta')x +
	(\varpi^{-(\scdepth + 1)}s\theta)x\inv
\bigr)
= \scPhi_{\varpi^{\scdepth + 1}}(u x + v x\inv),
\]
where \((u, v)\) is as in Notation \ref{notn:u-v-and-m}, so
\begin{equation}
\tag{$\dag$}
\label{eq:Bessel}
\Pint_{\field\mult}
	\bessel(\theta', \theta; x)\chi(x)
\textup d\mult x
= \Bessel_\chi(u, v)
\end{equation}
for
\(\chi \in \widehat{\field\mult}\).

Now note that \(\tfrac1 2(1 + \sgn_\theta)\) is the
characteristic function of \(\Norm_\theta(k_\theta\mult)\), so
we may re-write \eqref{eq:orbital-as-Pint} as
\begin{equation}
\tag{$**$}
\label{eq:first-orbital}
q^{-(\scdepth + 1)/2}
\Pint_{\field\mult}
	\tfrac1 2(1 + \sgn_\theta(x))\dotm
	\bessel(\theta', \theta; x)
	\mc H(\scPhi, -s x\inv)
\textup d\mult x.
\end{equation}

By \cite{shalika:thesis}*{Lemma 1.3.2}
and Lemma \ref{lem:G-facts}, we have
\[
\mc H(\scPhi, b) = \abs b^{-1/2}\begin{cases}
\sgn_\varpi(b)G_\varpi(\scPhi), &
	\text{\(\scdepth - \ord(b)\) even} \\
1, &
	\text{\(\scdepth - \ord(b)\) odd.}
\end{cases}
\]

We find it useful to offer a description of
\(\mc H(\scPhi, b)\) without explicit use of cases.
As above, we note that
\(\tfrac1 2(1 + (-1)^n\sgn_\epsilon)\) is the characteristic
function of
\(\set{b \in \field\mult}{\ord(b) \equiv n \pmod2}\),
so that we may re-write
\[
\mc H(\scPhi, b)
= \tfrac1 2(
	1 + (-1)^\scdepth\sgn_\epsilon(b)
)\sgn_\varpi(b)G_\varpi(\scPhi) +
\tfrac1 2(
	1 - (-1)^\scdepth\sgn_\epsilon(b)
).
\]
Plugging this into \eqref{eq:first-orbital},
with \(b = -s t\inv\), gives
\begin{align*}
M^G_{X^*}(Y)
= \tfrac1 2\abs s^{-1/2}q^{-(\scdepth + 1)/2}
\Pint_{\field\mult}
	& \tfrac1 2(1 + \sgn_\theta(x))\times{} \\
	& \bigl(
		(1 - \gammaun(s)\sgn_\epsilon(x))
			\gammaram(s)
			\sgn_\varpi(x) +{} \\
	& \qquad (1 + \gammaun(s)\sgn_\epsilon(x))
	\bigr)\times{} \\
	& \abs x^{1/2}
	\bessel(\theta', \theta; x)
\textup d\mult x.
\end{align*}
Expanding the product and applying \eqref{eq:Bessel}
gives the desired formula.
\end{proof}

\subsection{`Deep' Bessel functions}

By Proposition \ref{prop:second-orbital},
one approach to computing \(M^G_{X^*}(Y)\)
(hence \(\hat\mu^G_{X^*}(Y)\), by
Proposition \ref{prop:orbital-integral-integral})
is to evaluate many Bessel functions, and this is
exactly what we do.
As Theorem
\ref{thm:sally-taibleson:special:bessel} makes
clear, the behaviour of Bessel functions is
more predictable when \(m < 2\) than
otherwise.
We introduce a convenient, but temporary, shorthand for
referring to Bessel functions in this
range; we will only use it in this section,
and \S\S\ref{sec:close-spun} and \ref{sec:close-ram}.

\begin{notn}
\label{notn:Bessel-abbrev}
We define
\[
[A; B]_{\theta, \scdepth}(\theta')
\ldef \abs\theta^{1/2}A +
	q^{-(\scdepth + 1)}\abs{D_\gg(Y)}^{-1/2}B(\theta').
\]
We will usually suppress the subscript on \([A; B]\),
and will sometimes write
\[
[
	A; B(1), B(\epsilon), B(\varpi), B(\epsilon\varpi)
](\theta')
\]
for the same quantity.
\end{notn}

\begin{prop}
\label{prop:Bessels}
With the notation of
Notations
\ref{notn:depth-and-scdepth},
\ref{notn:Y},
and
\ref{notn:u-v-and-m},
and Definition \ref{defn:Wald-i},
if \(m < 2\), then
\begin{multline*}
\abs s^{-1/2}q^{-(\scdepth + 1)/2}
\Bessel_{\nu^{1/2}\chi}(u, v) \\
= \begin{cases}
\bigl[Q_3(q^{-1/2}); 1\bigr](\theta'),
	& \chi = 1 \\
\gammaun(s)
\bigl[
	\sgn_\epsilon(\theta)Q_3(-q^{-1/2});
	\sgn_\epsilon
\bigr](\theta'),
	& \chi = \sgn_\epsilon \\
\gammaram(s)\inv
\bigl[
	\sgn_\varpi(\theta)q\inv;
	\sgn_\varpi
\bigr](\theta'),
	& \chi = \sgn_\varpi \\
-\gammaun(s)\gammaram(s)\inv
\bigl[
	\sgn_{\epsilon\varpi}(\theta)q\inv;
	\sgn_{\epsilon\varpi}
\bigr](\theta'),
	& \chi = \sgn_{\epsilon\varpi},
\end{cases}
\end{multline*}
where
\[
Q_3(T) = -T(T^2 + T + 1).
\]
\end{prop}

The unexpected factor \(\abs s^{-1/2}q^{-(\scdepth + 1)/2}\)
above crops up repeatedly in calculations (see, for example,
Proposition \ref{prop:second-orbital}), so it simplifies
matters to include it in this calculation.

\begin{proof}
By Theorem
\ref{thm:sally-taibleson:special:bessel}
and Lemma \ref{lem:Y-facts},
\begin{align*}
\Bessel_{\nu^{1/2}\chi}(u, v)
={} &
(\nu^{1/2}\chi)(v)\Gamma(\nu^{-1/2}\chi) +
(\nu^{-1/2}\chi)(u)\Gamma(\nu^{1/2}\chi) \\
={} &
(\nu^{1/2}\chi)(v\theta\inv)\times{} \\
&\qquad\bigl(
	(\nu^{1/2}\chi)(\theta)
		\Gamma(\nu^{-1/2}\chi) +
	(\nu^{-1/2}\chi)(u v\theta\inv)
		\Gamma(\nu^{1/2}\chi)
\bigr) \\
={} &
\abs s^{1/2}q^{(\scdepth + 1)/2}
\chi(\varpi^{\scdepth + 1}s)\bigl[
	\chi(\theta)\Gamma(\nu^{-1/2}\chi);
	\Gamma(\nu^{1/2}\chi)\dotm\chi
\bigr](\theta')
\end{align*}
whenever \(\chi^2 = 1\).

In particular, upon using
\cite{sally-taibleson:special}*{%
	Theorem 3.1(i, ii)%
} to compute the \(\Gamma\)-factors, we see that
\(\abs s^{-1/2}q^{-(\scdepth + 1)/2}
\Bessel_{\nu^{1/2}\chi}(u, v)\)
is given by
\begin{equation}
\tag{$*$}
\label{eq:complicated-Bessels}
\begin{cases}
\bigl[Q_3(q^{-1/2}); 1\bigr](\theta'),
	& \chi = 1 \\
\gammaun(s)
\bigl[
	\sgn_\epsilon(\theta)Q_3(-q^{-1/2});
	\sgn_\epsilon
\bigr](\theta'),
	& \chi = \sgn_\epsilon \\
\sgn_\varpi(\varpi^{\scdepth + 1}s)
\Gamma(\nu^{1/2}\sgn_\varpi)
\bigl[
	\sgn_\varpi(\theta)q\inv;
	\sgn_\varpi
\bigr](\theta'),
	& \chi = \sgn_\varpi \\
\gammaun(s)
\sgn_\varpi(\varpi^{\scdepth + 1}s)
\Gamma(\nu^{1/2}\sgn_{\epsilon\varpi})
\bigl[
	\sgn_{\epsilon\varpi}(\theta)q\inv;
	\sgn_{\epsilon\varpi}
\bigr](\theta'),
	& \chi = \sgn_{\epsilon\varpi}.
\end{cases}
\end{equation}
By Theorem 3.1(ii) \loccit again,
and the fact that
\(\sgn_{\epsilon\varpi} = \nu^{i\pi/\ln(q)}\sgn_\varpi\),
we have that
\(\Gamma(\nu^{1/2}\sgn_{\epsilon\varpi})
= -\Gamma(\nu^{1/2}\sgn_\varpi)\);
and,
by Lemma \ref{lem:Gamma-varpi},
Definition \ref{defn:Wald-i},
and \eqref{eq:gamma-square},
\begin{align*}
& \sgn_\varpi(\varpi^{\scdepth + 1}s)
\Gamma(\nu^{1/2}\sgn_\varpi) \\
& \qquad= \sgn_\varpi(-1)^{\scdepth + 1}
\sgn_\varpi(s)\dotm
\sgn_\varpi(-1)^{\scdepth + 1}
G_\varpi(\scPhi) \\
& \qquad= \sgn_\varpi(s)G_\varpi(\scPhi) \\
& \qquad= \gammaram(s)\inv.
\end{align*}
This shows that \eqref{eq:complicated-Bessels} reduces to
the table in the statement.
\end{proof}

\section{Split and unramified orbital integrals}
\label{sec:orbital-spun}

Throughout this section, we have
\begin{equation}
\label{eq:theta-and-r-spun}
\theta = 1\text{ or }\theta = \epsilon\text,
	\quad\text{so that}\quad
\scdepth = r.
\end{equation}

In the split case,
\(\Bessel_\chi^1 = \Bessel_\chi\)
for \(\chi \in \widehat{\field\mult}\),
so Proposition \ref{prop:second-orbital} gives
\begin{equation}
\label{eq:third-orbital-split}
\begin{aligned}
M^G_{X^*}(Y)
={} &
\tfrac1 2\abs s^{-1/2}q^{-(r + 1)/2}\times{} \\
&\quad\Bigl(
	\bigl(
		\Bessel_{\nu^{1/2}}(u, v) +
		\gammaun(s)\Bessel_{\nu^{1/2}\sgn_\epsilon}(u, v)
	\bigr) +{} \\
&\quad\qquad
	\gammaram(s)\bigl(
		\Bessel_{\nu^{1/2}\sgn_\varpi}(u, v) -
		\gammaun(s)\Bessel_{\nu^{1/2}\sgn_{\epsilon\varpi}}(u, v)
	\bigr)
\Bigr),
\end{aligned}
\end{equation}
In the unramified case,
\(\Bessel_\chi^\epsilon = \Bessel_{\chi\sgn_\epsilon}^\epsilon\)
for \(\chi \in \widehat{\field\mult}\),
so Proposition \ref{prop:second-orbital} gives
\begin{equation}
\label{eq:third-orbital-un}
\begin{aligned}
M^G_{X^*}(Y)
={} &
\tfrac 1 2\abs s^{-1/2}q^{-(r + 1)/2}\times{} \\
&\quad\bigl(
	(1 + \gammaun(s))
		\Bessel_{\nu^{1/2}}^\epsilon(u, v) +
	\gammaram(s)(1 - \gammaun(s))
		\Bessel_{\nu^{1/2}\sgn_\varpi}^\epsilon(u, v)
\bigr).
\end{aligned}
\end{equation}

By \eqref{eq:sgn-and-G} and
\eqref{eq:gamma-square},
\begin{equation}
\label{eq:sgn-and-G-spun}
\sgn_\varpi(v)G_\varpi(\scPhi_{\varpi^{r + 1}})
= \begin{cases}
\sgn_\varpi(-1)\gammaram(s) = \gammaram(s)\inv, &
	\theta = 1 \\
\sgn_\varpi(-\epsilon)\gammaram(s)
	= -\gammaram(s)\inv, &
	\theta = \epsilon.
\end{cases}
\end{equation}

\subsection{Far from zero}
\label{sec:shallow-spun}

The results of this section are special cases for split and
unramified orbital integrals of results of Waldspurger
\cite{waldspurger:loc-trace-form}*{Proposition VIII.1}.
We shall prove analogues of these results for ramified
orbital integrals in \S\ref{sec:shallow-ram}.

The qualitative behaviour of unramified orbital integrals
does not change as we pass from elements of depth less than
\(r\) to those of depth exactly \(r\); this is unlike the
situation for ramified orbital integrals.  See
\S\ref{sec:bad-ram}.

\begin{thm}
\label{thm:vanish-spun}
If
\(\depth(X^*) + \depth(Y) \le 0\)
and
\(X^*\) is split or unramified,
then
\(M^G_{X^*}(Y) = 0\) unless
\(X^*\) and \(Y\) lie in \(G\)-conjugate tori.
\end{thm}

\begin{proof}
Recall that \(\theta = 1\) if
\(X^*\) is split,
and
\(\theta = \epsilon\) if
\(X^*\) is unramified.

By \eqref{eq:m-and-depth-Y}, \(m \ge 2\);
in fact,
\(m > 2\) (indeed, \(m\) is odd)
unless \(\ord(\theta')\) is even.

If \(m > 2\), then Proposition \ref{prop:Bessel-shallow}
and \eqref{eq:uv-is-square}
show that \(M^G_{X^*}(Y) = 0\)
unless \(\theta\theta' \in (\field\mult)^2\).
By \S\ref{sec:tori}, it therefore suffices
to consider the cases when
\(\theta = \epsilon\) and
\(\theta' = \varpi^2\epsilon\),
i.e., \(X^*\) and \(Y\) lie in stably, but not rationally,
conjugate tori;
and when \(m = 2\) and
\(\sset{\theta, \theta'}
= \sset{1, \epsilon}\),
i.e., one of \(X^*\) or \(Y\) is split, and the other
unramified.

Suppose first that
\(\theta = \epsilon\) and
\(\theta' = \varpi^2\epsilon\), so that
\(\ord(u) = \ord(v) + 2\).
By Corollary \ref{cor:Bessel-twist},
\eqref{eq:third-orbital-un} becomes
\(M^G_{X^*}(Y) = 0\).

Now suppose that
\(\sset{\theta, \theta'} = \sset{1, \epsilon}\)
and \(m = 2\).
By Corollary \ref{cor:Bessel-twist},
since \(\ord(u) = \ord(v)\),
\[
\Bessel_{\nu^{1/2}}(u, v)
= \Bessel_{\nu^{1/2}\sgn_\epsilon}(u, v)
\qandq
\Bessel_{\nu^{1/2}\sgn_\varpi}(u, v)
= \Bessel_{\nu^{1/2}\sgn_{\epsilon\varpi}}(u, v),
\]
so \eqref{eq:third-orbital-split}
agrees with \eqref{eq:third-orbital-un}.
We shall work with
\eqref{eq:third-orbital-un}, since it is
simpler.

By Corollary \ref{cor:Bessel-Kloost} and
\eqref{eq:uv-is-square},
\(\Bessel_{\nu^\alpha\sgn_\varpi}(u, v) = 0\)
for all \(\alpha \in \C\); in particular,
for
\(\alpha = 1/2\)
and
\(\alpha = 1/2 + i\pi/\ln(q)\).
By \eqref{eq:m-and-depth-Y}, \(\ord(s) = r\), so,
by Definition \ref{defn:Wald-i},
\(\gammaun(s) = -1\),
and \eqref{eq:third-orbital-un}
(hence also \eqref{eq:third-orbital-split})
becomes
\[
M^G_{X^*}(Y)
= \tfrac1 2\abs s^{-1/2}
\Bessel_{\nu^{1/2}\sgn_\varpi}^\epsilon(u, v)
= 0.\qedhere
\]
\end{proof}

\begin{thm}
\label{thm:shallow-spun}
If
\(\depth(X^*) + \depth(Y) \le 0\)
and
\(X^*\) and \(Y\) lie in a common split or unramified torus \bT
(with \(T = \bT(k)\)),
then
\[
M^G_{X^*}(Y)
= q^{-(r + 1)}\abs{D_\gg(Y)}^{-1/2}
\gamma_\Phi(X^*, Y)
\sum_{\sigma \in W(G, T)}
	\Phi(\langle\Ad^*(\sigma)X^*, Y\rangle),
\]
where
\(\gamma_\Phi(X^*, Y)\) is as in
Definition \ref{defn:Wald-i}.
\end{thm}

\begin{proof}
The hypothesis implies that
\(\theta = \theta'\), so \(u = v\).
By Corollary \ref{cor:Bessel-twist},
\[
\Bessel_{\nu^{1/2}}(u, v)
= \Bessel_{\nu^{1/2}\sgn_\epsilon}(u, v)
\qandq
\Bessel_{\nu^{1/2}\sgn_\varpi}(u, v)
= \Bessel_{\nu^{1/2}\sgn_{\epsilon\varpi}}(u, v),
\]
so \eqref{eq:third-orbital-split} agrees
with \eqref{eq:third-orbital-un}.
We shall work with
\eqref{eq:third-orbital-un}, since it is
simpler.

By Remark \ref{rem:Weyl},
\(W(G, T_\theta)
= \sset{1, \sigma_\theta}\),
where \(\Ad^*(\sigma_\theta)X^* = -X^*\).

We may take the square root \(w\) of \(u v\) in
Proposition \ref{prop:Bessel-shallow}
to be just \(u\).
By \eqref{eq:m-and-depth-Y},
\begin{equation}
\tag{$*$}
\label{eq:m-shallow-spun}
q^{-m/4} = q^{-(r + 1)/2}q^{\ord(s)/2}
= q^{-(r + 1)/2}\abs s^{-1/2}.
\end{equation}
By Notations \ref{notn:depth-and-scdepth}
and \ref{notn:u-v-and-m},
\begin{equation}
\tag{$**$}
\label{eq:new-Phi-at-w-shallow-spun}
\scPhi_{\varpi^{r + 1}}(\pm2w)
= \scPhi(\pm2s\theta)
= \Phi(\pm2\beta s\theta)
= \Phi(\pm\langle X^*, Y\rangle)
\end{equation}
(the last equality following, for example, from
\eqref{eq:hilarious-trace}).

Suppose that
\(\ord(s) \not\equiv r \pmod2\),
so that \(\gammaun(s) = 1\) and \(\gamma_\Phi(X^*, Y) = 1\).
By Corollary \ref{cor:Bessel-twist},
since \(u = v\),
\eqref{eq:third-orbital-un}
(hence also \eqref{eq:third-orbital-split})
becomes
\begin{equation}
\tag{$\dag$}
\label{eq:orbital-not-r-shallow-spun}
\begin{aligned}
M^G_{X^*}(Y)
& {}= \tfrac1 2\abs s^{-1/2}q^{-(r + 1)/2}\dotm
2\dotm\Bessel_{\nu^{1/2}}^\epsilon(u, v) \\
& {}= \abs s^{-1/2}q^{-(r + 1)/2}
\Bessel_{\nu^{1/2}}(u, v).
\end{aligned}
\end{equation}
Since
\(m > 2\) and \(4 \mid m\)
by \eqref{eq:m-and-depth-Y},
combining
Proposition \ref{prop:Bessel-shallow},
\eqref{eq:m-shallow-spun},
and
\eqref{eq:new-Phi-at-w-shallow-spun}
gives
\begin{equation}
\tag{$\dag\dag$}
\label{eq:Bessel-1-shallow-spun}
\begin{aligned}
\Bessel_{\nu^{1/2}}(u, v)
& {}= q^{-(r + 1)/2}\abs s^{-1/2}\bigl(
	\Phi(\langle X^*, Y\rangle) +
	\Phi(-\langle X^*, Y\rangle)
\bigr) \\
& {}= q^{-(r + 1)/2}\abs s^{-1/2}
\sum_{\sigma \in W(G, T_\theta)}
	\Phi(\langle\Ad^*(\sigma)X^*, Y\rangle) \\
& {}= q^{-(r + 1)/2}\abs{s\theta'}^{-1/2}
\gamma_\Phi(X^*, Y)
\sum_{\sigma \in W(G, T_\theta)}
	\Phi(\langle\Ad^*(\sigma)X^*, Y\rangle).
\end{aligned}
\end{equation}
The result (in this case)
now follows from Lemma \ref{lem:Y-facts}
by combining
\eqref{eq:orbital-not-r-shallow-spun}
and
\eqref{eq:Bessel-1-shallow-spun}.

Suppose now that \(\ord(s) \equiv r \pmod2\),
so that \(\gammaun(s) = -1\) and
\[
\gamma_\Phi(X^*, Y)
= \begin{cases}
1,  & \theta = 1         \\
-1, & \theta = \epsilon.
\end{cases}
\]
Again by Corollary \ref{cor:Bessel-twist},
since \(u = v\),
\eqref{eq:third-orbital-un}
(hence also \eqref{eq:third-orbital-split})
becomes
(as in
\eqref{eq:orbital-not-r-shallow-spun})
\begin{equation}
\tag{$\dag'$}
\label{eq:orbital-r-shallow-spun}
M^G_{X^*}(Y)
= \abs s^{-1/2}q^{-(r + 1)/2}\gammaram(s)
\Bessel_{\nu^{1/2}\sgn_\varpi}(u, v).
\end{equation}
Since
\(4 \nmid m\) by \eqref{eq:m-and-depth-Y},
if \(m > 2\), then combining
Proposition \ref{prop:Bessel-shallow},
\eqref{eq:m-shallow-spun},
\eqref{eq:sgn-and-G-spun},
and
\eqref{eq:new-Phi-at-w-shallow-spun}
gives
(as in \eqref{eq:Bessel-1-shallow-spun})
\begin{equation}
\tag{$\dag\dag'_{< r}$}
\label{eq:Bessel-varpi-shallow-spun}
\begin{aligned}
& \Bessel_{\nu^{1/2}\sgn_\varpi}(u, v) \\
& \qquad= q^{-(r + 1)/2}\abs{s\theta'}^{-1/2}
\gammaram(s)\inv\gamma_\Phi(X^*, Y)
\sum_{\sigma \in W(G, T_\theta)}
	\Phi(\langle\Ad^*(\sigma)X^*, Y\rangle).
\end{aligned}
\end{equation}
If \(m = 2\), then, by
Lemma \ref{lem:Y-facts},
\eqref{eq:ord-u},
and
\eqref{eq:m-and-depth-Y},
\(\abs s = q^{-r}\)
and
\(\ord(u) = -1\).
Thus, combining
Corollary \ref{cor:Bessel-Kloost},
\eqref{eq:sgn-and-G-spun},
and
\eqref{eq:new-Phi-at-w-shallow-spun}
gives
\begin{equation}
\tag{$\dag\dag'_{= r}$}
\label{eq:Bessel-varpi-bad-spun}
\begin{aligned}
&\Bessel_{\nu^{1/2}\sgn_\varpi}(u, v) \\
& \qquad= q^{-1/2}
\gammaram(s)\inv\gamma_\Phi(X^*, Y)
\sum_{\sigma \in W(G, T_\theta)}
	\Phi(\langle\Ad^*(\sigma)X^*, Y\rangle) \\
& \qquad= q^{-(r + 1)/2}\abs{s\theta'}^{-1/2}
\gammaram(s)\inv\gamma_\Phi(X^*, Y)
\sum_{\sigma \in W(G, T_\theta)}
	\Phi(\langle\Ad^*(\sigma)X^*, Y\rangle).
\end{aligned}
\end{equation}
The result (in this case) now follows
by combining
\eqref{eq:orbital-r-shallow-spun},
	\eqref{eq:Bessel-varpi-shallow-spun}
	\emph{or}
	\eqref{eq:Bessel-varpi-bad-spun},
and
Lemma \ref{lem:Y-facts}
\end{proof}

\subsection{Close to zero}
\label{sec:close-spun}

\begin{thm}
\label{thm:close-spun}
If
\(\depth(X^*) + \depth(Y) > 0\),
and
\(X^*\) is split or unramified,
then let
\(\gamma_\Phi(X^*, Y)\) and \(c_0(X^*)\)
be as in Definitions \ref{defn:Wald-i} and \ref{defn:const},
respectively.
Then
%
\[
M^G_{X^*}(Y)
= c_0(X^*) +
\frac2{n(X^*)}q^{-(r + 1)}\abs{D_\gg(Y)}^{-1/2}
	\gamma_\Phi(X^*, Y),
\]
where
\[
n(X^*) = \begin{cases}
1, & \text{\(X^*\) split} \\
2, & \text{\(X^*\) elliptic.}
\end{cases}
\]
\end{thm}

\begin{proof}
By \eqref{eq:m-and-depth-Y}, \(m < 2\).

By Proposition \ref{prop:Bessels}, using
Notation \ref{notn:Bessel-abbrev},
\eqref{eq:third-orbital-split} becomes
\[
M^G_{X^*}(Y)
= \tfrac1 2\bigl[
	Q_3(q^{-1/2}) +
	Q_3(q^{-1/2}) -
	q\inv -
	q\inv;
	1 +
	\sgn_\epsilon +
	\sgn_\varpi +
	\sgn_{\epsilon\varpi}
\bigr](\theta').
\]
Since
\begin{equation}
\label{eq:inv-Q3}
Q_3(q^{-1/2}) + Q_3(-q^{-1/2})
= -2T^2\bigr|_{T = q^{-1/2}}
= -2q\inv,
\end{equation}
this simplifies (by the Plancherel formula on
\(\field\mult/(\field\mult)^2\)!)\ to
\[
M^G_{X^*}(Y)
= [-2q\inv; 2, 0, 0, 0].
\]
Similarly, \eqref{eq:third-orbital-un} becomes
\begin{align*}
M^G_{X^*}(Y)
= \tfrac1 2\Bigl(&
\tfrac1 2(1 + \gammaun(s))\underbrace{\bigl[
	Q_3(q^{-1/2}) + \gammaun(s)Q_3(-q^{-1/2});
	1 + \gammaun(s)\sgn_\epsilon
\bigr]}_{\text{(I)}} + \\
&\qquad\tfrac1 2(1 - \gammaun(s))\underbrace{\bigl[
	-(1 - \gammaun(s))q\inv;
	(1 - \gammaun(s)\sgn_\epsilon)\sgn_\varpi
\bigr]}_{\text{(II)}}
\Bigr)(\theta').
\end{align*}
Since \(\gammaun(s) = \pm1\)
(see Definition \ref{defn:Wald-i}),
we may replace \(\gammaun(s)\) by \(1\) in (I)
and by \(-1\) in (II),
then use \eqref{eq:inv-Q3}
and check case-by-case to see that the formula simplifies to
\[
M^G_{X^*}(Y)
= [-q\inv; 1, \gammaun(s), 0, 0](\theta').\qedhere
\]
\end{proof}

\section{Ramified orbital integrals}
\label{sec:orbital-ram}

Throughout this section, we have
\begin{equation}
\theta = \varpi,
	\quad\text{so that}\quad
\scdepth = r + \tfrac1 2 \rdef \indexmem h.
\end{equation}
Then
\(\Bessel_\chi^\varpi = \Bessel_{\chi\sgn_\varpi}^\varpi\)
for \(\chi \in \widehat{\field\mult}\), so
Proposition \ref{prop:second-orbital} gives
\begin{equation}
\label{eq:third-orbital-ram}
M^G_{X^*}(Y)
= \tfrac1 2\abs s^{-1/2}\bigl(
	(1 + \gammaram(s))
		\Bessel_{\nu^{1/2}}^\varpi(u, v) +
	\gammaun(s)(1 - \gammaram(s))
		\Bessel_{\nu^{1/2}\sgn_\epsilon}^\varpi(u, v)
\bigr).
\end{equation}

By \eqref{eq:sgn-and-G},
\begin{equation}
\label{eq:sgn-and-G-ram}
\sgn_\varpi(v)G_\varpi(\scPhi_{\varpi^{h + 1}}) \\
= \sgn_\varpi(-\varpi)\gammaram(s)
= \gammaram(s).
\end{equation}

\subsection{Far from zero}
\label{sec:shallow-ram}

As in \S\ref{sec:shallow-spun},
the results of this section are special cases of
\cite{waldspurger:loc-trace-form}*{Proposition VIII.1}.

\begin{thm}
\label{thm:vanish-ram}
If
\(\depth(X^*) + \depth(Y) < 0\)
and
\(X^*\) is ramified,
then \(M^G_{X^*}(Y) = 0\) unless
\(X^*\) and \(Y\) lie in \(G\)-conjugate tori.
\end{thm}

\begin{proof}
By \eqref{eq:m-and-depth-Y}, \(m > 2\), so
Proposition \ref{prop:Bessel-shallow} and \eqref{eq:uv-is-square}
show that
\(M^G_{X^*}(Y) = 0\)
unless \(\varpi\theta' \in (\field\mult)^2\).
By \S\ref{sec:tori}, it therefore suffices
to consider the case
when \(-1 \in (\resfld\mult)^2\)
(so \(\sgn_\varpi(-1) = 1\))
and \(\theta' = \epsilon^2\varpi\),
i.e., \(X^*\) and \(Y\) lie in stably, but not rationally,
conjugate tori.

By \eqref{eq:uv}, we may take the square root \(w\) of \(u v\)
to be \(w = \varpi^{-h}s\epsilon = \epsilon\inv u\).
Then \(u\inv w = \epsilon\inv\), so
Proposition \ref{prop:Bessel-shallow} shows
(whether or not \(4\) divides \(m\))
that, if \(\chi\) is mildly ramified \emph{and} trivial at
\(-1\), then
\[
\Bessel_{\chi\sgn_\varpi}(u, v)
= \sgn_\varpi(u\inv\varpi)\Bessel_\chi(u, v)
= -\Bessel_\chi(u, v),
\]
hence that
\(\Bessel_\chi^\varpi(u, v) = 0\).
In particular, this equality holds for
\(\chi = \nu^{1/2}\)
and
\(\chi = \nu^{1/2}\sgn_\epsilon\).
It follows from \eqref{eq:third-orbital-ram} that
\(M^G_{X^*}(Y) = 0\).
\end{proof}

\begin{thm}
\label{thm:shallow-ram}
If
\(\depth(X^*) + \depth(Y) < 0\),
and
\(X^*\) and \(Y\) lie in a common ramified torus \bT
(with \(T = \bT(k)\)),
then
\[
M^G_{X^*}(Y)
= q^{-(h + 1)}\abs{D_\gg(Y)}^{-1/2}\gamma_\Phi(X^*, Y)
\sum_{\sigma \in W(G, T)}
	\Phi(\langle\Ad^*(\sigma)(X^*), Y\rangle),
\]
where \(\gamma_\Phi(X^*, Y)\) is as in
Definition \ref{defn:Wald-i}.
\end{thm}

\begin{proof}
Since we have fixed \(\theta = \varpi\),
the hypothesis implies that \(\theta' = \varpi\).
In particular, \(u = v\).
Write \(\sigma_\varpi\) for the non-trivial element of
\(W(\bG, \bT_\varpi)(\field_\varpi)\),
so that \(\Ad^*(\sigma_\varpi)X^* = -X^*\).
Note that it is possible that
\(\sigma_\varpi\) is not \(\field\)-rational.
More precisely, by \S\ref{sec:tori}, we have that
\[
W(G, T_\varpi) = \begin{cases}
\sset{1, \sigma_\varpi}, &
	\sgn_\varpi(-1) = 1    \\
\sset1, &
	\sgn_\varpi(-1) = -1.
\end{cases}
\]

By \eqref{eq:m-and-depth-Y},
\begin{equation}
\tag{$*$}
\label{eq:m-shallow-ram}
q^{-m/4} = q^{-(h - \ord(s))/2} = q^{-h/2}\abs s^{-1/2}.
\end{equation}

By Corollary \ref{cor:Bessel-twist},
since \(u = v\),
\[
\Bessel_{\nu^{1/2}}^\varpi(u, v)
= \Bessel_{\nu^{1/2}\sgn_\epsilon}^\varpi(u, v),
\]
so \eqref{eq:third-orbital-ram} becomes
\begin{equation}
\tag{$\dag$}
\label{eq:orbital-shallow-ram}
M^G_{X^*}(Y)
= \tfrac1 2\abs s^{-1/2}q^{-(h + 1)/2}
\bigl(
	(1 + \gammaram(s)) + \gammaun(s)(1 - \gammaram(s))
\bigr)\Bessel_{\nu^{1/2}}^\varpi(u, v).
\end{equation}
It remains to compute
\(\Bessel_{\nu^{1/2}}^\varpi(u, v)\).

We will use Proposition \ref{prop:Bessel-shallow}, but, for
simplicity, we want to avoid splitting into cases depending
on whether or not \(4 \mid m\).
By \eqref{eq:m-and-depth-Y},
the restrictions to \(\field \setminus \pp^{h - 1}\) of
\(\tfrac1 2(1 + (-1)^h\sgn_\epsilon)
= \tfrac1 2(1 - \gammaun)\)
and
\(\tfrac1 2(1 + \gammaun)\)
are characteristic functions that indicate whether
\(4 \mid m\) or \(4 \nmid m\),
respectively.
(We omit \(\pp^{h - 1}\) because we are concerned with
the case where \(\depth(Y) < r\), so that
\(\ord(s) < r - \tfrac1 2 = h - 1\).)

Thus, if \(\sgn_\varpi(-1) = -1\),
then combining
Proposition \ref{prop:Bessel-shallow},
\eqref{eq:m-shallow-ram}, and
\eqref{eq:sgn-and-G-ram}
gives
\begin{equation}
\tag{${*{*}*}\textsub{ns}$}
\label{eq:basic-Bessel-ram}
\begin{aligned}
\Bessel_{\nu^\alpha}(u, v)
={} &q^{-h/2}\abs s^{-1/2}\times{} \\
& \qquad\Bigl(
	\tfrac1 2\bigl[
		(1 - \gammaun(s)) +
		(1 + \gammaun(s))\gammaram(s)
	\bigr]\times{} \\
& \qquad\qquad\scPhi_{\varpi^{h + 1}}(2\varpi^{-h}s)
	\overset{\text{(\S)}}+{} \\
& \qquad
	\tfrac1 2\bigl[
		(1 - \gammaun(s)) \overset{\text{(\P)}}-
		(1 + \gammaun(s))\gammaram(s)
	\bigr]\times{} \\
& \qquad\qquad\scPhi_{\varpi^{h + 1}}(-2\varpi^{-h}s)\Bigr) \\
={} &\tfrac1 2 q^{-(h + 1)/2}\abs{s\theta'}^{-1/2}\times{} \\
& \qquad\Bigl(
	\bigl[
		(1 + \gammaram(s)) -
		\gammaun(s)(1 - \gammaram(s))
	\bigr]\Phi(\langle X^*, Y\rangle) +{} \\
& \qquad\qquad
	\bigl[
		(1 - \gammaram(s)) -
		\gammaun(s)(1 + \gammaram(s))
	\bigr]\Phi(\langle\Ad^*(\sigma_\varpi)X^*, Y\rangle)
	\Bigr)
\end{aligned}
\end{equation}
and (changing the sign at (\S), but not at (\P)) that
\begin{align*}
\Bessel_{\nu^\alpha\sgn_\varpi}(u, v)
={} &\tfrac1 2 q^{-(h + 1)/2}\abs{s\theta'}^{-1/2}\times{} \\
& \qquad\Bigl(\bigl[
		(1 + \gammaram(s)) -
		\gammaun(s)(1 - \gammaram(s))
	\bigr]\Phi(\langle X^*, Y\rangle) -{} \\
& \qquad\qquad\bigl[
		(1 - \gammaram(s)) -
		\gammaun(s)(1 + \gammaram(s))
	\bigr]\Phi(\langle\Ad^*(\sigma_\varpi)X^*, Y\rangle)\Bigr),
\end{align*}
so that
\begin{equation}
\tag{$\ddag\textsub{ns}$}
\label{eq:hat-Bessel-without-i}
\Bessel_{\nu^\alpha}^\varpi(u, v)
= \tfrac1 2 q^{-(h + 1)/2}\abs{s\theta'}^{-1/2}
\bigl[
	(1 + \gammaram(s)) -
	\gammaun(s)(1 - \gammaram(s))
\bigr]\Phi(\langle X^*, Y\rangle).
\end{equation}
Similarly, if \(\sgn_\varpi(-1) = 1\), then
(changing the sign at (\P), but not at (\S), in
\eqref{eq:basic-Bessel-ram})
we obtain
\begin{equation}
\tag{${*{*}*}\textsub s$}
\begin{aligned}
\Bessel_{\nu^\alpha}(u, v)
={} &
\Bessel_{\nu^\alpha\sgn_\varpi}(u, v) \\
={} &
\tfrac1 2 q^{-(h + 1)/2}\abs{s\theta'}^{-1/2}\bigl[
	(1 + \gammaram(s)) -
	\gammaun(s)(1 - \gammaram(s))
\bigr]\times{} \\
&\qquad\bigl[
	\Phi(\langle X^*, Y\rangle) +
	\Phi(\langle\Ad^*(\sigma_\varpi)X^*, Y\rangle)
\bigr],
\end{aligned}
\end{equation}
so that
\begin{multline}
\tag{$\ddag\textsub s$}
\label{eq:hat-Bessel-with-i}
\Bessel_{\nu^\alpha}^\varpi(u, v)
= \Bessel_{\nu^\alpha}(u, v)
= \tfrac1 2 q^{-(h + 1)/2}\abs{s\theta'}^{-1/2}\bigl[
	(1 + \gammaram(s)) -
	\gammaun(s)(1 - \gammaram(s))
\bigr]\times{} \\
\bigl[
	\Phi(\langle X^*, Y\rangle) +
	\Phi(\langle\Ad^*(\sigma_\varpi)X^*, Y\rangle)
\bigr].
\end{multline}
We may write
\eqref{eq:hat-Bessel-without-i}
and
\eqref{eq:hat-Bessel-with-i}
uniformly as
\begin{multline}
\tag{$\ddag$}
\label{eq:hat-Bessel-anyway}
\Bessel_{\nu^\alpha}^\varpi(u, v)
= \tfrac1 2q^{-(h + 1)/2}\abs{s\theta'}^{-1/2}\bigl[
	(1 + \gammaram(s)) -
	\gammaun(s)(1 - \gammaram(s))
\bigr]\times{} \\
\sum_{\sigma \in N_G(T_\varpi)/T_\varpi}
	\Phi(\langle\Ad^*(\sigma)X^*, Y\rangle).
\end{multline}

Upon combining \eqref{eq:orbital-shallow-ram},
\eqref{eq:hat-Bessel-anyway},
and Lemma \ref{lem:Y-facts}, we obtain
the desired formula by noting that
\begin{multline*}
\bigl[
	(1 + \gammaram(s)) +
	\gammaun(s)(1 - \gammaram(s))
\bigr]\dotm\bigl[
	(1 + \gammaram(s)) -
	\gammaun(s)(1 - \gammaram(s))
\bigr] \\
= (1 + \gammaram(s))^2 -
	\gammaun(s)^2(1 - \gammaram(s))^2
= 4\gammaram(s) = 4\gamma_\Phi(X^*, Y)
\end{multline*}
(since \(\gammaun(s)^2 = 1\)).
\end{proof}

\subsection{The bad shell}
\label{sec:bad-ram}

We shall be concerned in this section with the behaviour of
\(M^G_{X^*}\)
(hence \(\hat\mu^G_{X^*}\), by Proposition
\ref{prop:orbital-integral-integral})
at the `bad shell', i.e., on those regular, semisimple elements
\(Y\) such that \(\depth(Y) = r\).
We shall assume that this is the case throughout the
section.
By \eqref{eq:m-and-depth-Y}, this implies that
\(m = 2\) and that \(\ord(\theta')\) is odd, i.e., that
\(Y\) belongs to a ramified torus.
By \S\ref{sec:tori}, we can in fact assume that
\(\ord(\theta') = 1\).  Then, by
Lemma \ref{lem:Y-facts},
\begin{equation}
\label{eq:ord-s-bad-ram}
\ord(s) = h - 1
\implies
\sgn_\epsilon(s) = (-1)^{h - 1}
\text{ and }
\abs{s\theta'} = q^{-h}.
\end{equation}

By Definition \ref{defn:Wald-i}, the formula that holds in
the situation of Theorem \ref{thm:that-bad-ram} holds also,
suitably understood,
in the situation of Theorem \ref{thm:other-bad-ram}.
We find it useful to separate them anyway.

\begin{rem}
\label{rem:named-torus}
In this section only, we need to name the specific ramified
torus in which we are interested.
We therefore assume in
Theorems \ref{thm:other-bad-ram} and \ref{thm:that-bad-ram}
that \(X^* \in \ttt_\varpi^*\).
See Remark \ref{rem:what-about} for a discussion of how to
handle other ramified tori.
\end{rem}

\begin{thm}
\label{thm:other-bad-ram}
If
\(\depth(X^*) + \depth(Y) = 0\),
and
\(Y\) lies in a ramified torus that is \emph{not}
stably conjugate to \(\bT_\varpi\),
then
\[
M^G_{X^*}(Y)
= \tfrac1 2 q^{-(h + 1)}\dotm q^{-1/2}\abs{D_\gg(Y)}^{-1/2}
\sum_{Z \in (\ttt_\varpi)_{r:r{+}}}
	\Phi(\langle X^*, Z\rangle)\sgn_\varpi(Y^2 - Z^2),
\]
where we identify the scalar matrices \(Y^2\) and \(Z^2\)
with elements of \(\field\) in the natural way.
\end{thm}

\begin{proof}
By \S\ref{sec:tori}, it suffices to consider the case where
\(\theta' = \epsilon\varpi\).

By Corollary \ref{cor:Bessel-twist},
since \(\ord(u) = \ord(v)\),
\[
\Bessel_{\nu^{1/2}}^\varpi(u, v)
= \Bessel_{\nu^{1/2}\sgn_\epsilon}^\varpi(u, v);
\]
and, by Corollary \ref{cor:Bessel-Kloost} and
\eqref{eq:uv-is-square},
\(\Bessel_{\nu^{1/2}\sgn_\varpi}(u, v) = 0\), so
\[
\Bessel_{\nu^{1/2}}^\varpi(u, v)
= \tfrac1 2\Bessel_{\nu^{1/2}}(u, v);
\]
so, by \eqref{eq:third-orbital-ram} and
\eqref{eq:ord-s-bad-ram},
\begin{equation}
\tag{$*$}
\label{eq:orbital-other-bad-ram}
\begin{aligned}
M^G_{X^*}(Y)
={} &
\tfrac1 4\abs s^{-1/2}q^{-(h + 1)/2}\times{} \\
&\qquad\bigl(
	(1 + \gammaram(s)) - (-1)^h\sgn_\epsilon(s)(1 - \gammaram(s))
\bigr)\Bessel_{\nu^{1/2}}(u, v) \\
={} &
\tfrac1 4\abs s^{-1/2}q^{-(h + 1)/2}\dotm2\dotm
\Bessel_{\nu^{1/2}}(u, v) \\
={} &
\tfrac1 2\abs s^{-1/2}\Bessel_{\nu^{1/2}}(u, v).
\end{aligned}
\end{equation}

Finally, another application of
Corollary \ref{cor:Bessel-Kloost},
together with \eqref{eq:ord-u},
gives that
\[
\Bessel_{\nu^{1/2}}(u, v)
= q\inv\sum_{c \in \pp\inv/\pint}
	\scPhi_{\varpi^{h + 1}}(2c)
	\sgn_\varpi(c^2 - (\varpi^{-h}s)^2\epsilon)
\]
Replacing \(c\) by \(\varpi^{-h}c\),
and using \eqref{eq:ord-s-bad-ram} again,
allows us to re-write
\begin{equation}
\tag{$**$}
\label{eq:Bessel-other-bad-ram}
\Bessel_{\nu^{1/2}}(u, v)
= q^{-(h + 2)/2}\abs{s\theta'}^{-1/2}
\sum_{c \in \pp^{h - 1}/\pp^h}
	\Phi(2\beta\varpi c)
	\sgn_\varpi(c^2 - s^2\epsilon).
\end{equation}
By Definition \ref{defn:filt-and-depth},
the isomorphism \(c \mapsto c\dotm\sqrt\varpi\)
of \(\field\) with \(\ttt_\varpi\) identifies
\(\pp^{h - 1}/\pp^h\) with
\((\ttt_\varpi)_{(h - 1/2):(h + 1/2)}
= (\ttt_\varpi)_{r:r{+}}\).
If \(c\) is mapped to \(Z\), then
(by \eqref{eq:hilarious-trace}, for example)
\(2\beta\varpi c = \langle X^*, Z\rangle\),
and
\[
\sgn_\varpi(c^2 - s^2\epsilon)
= \sgn_\varpi(s^2\epsilon\varpi - c^2\varpi)
= \sgn_\varpi(Y^2 - Z^2).
\]
Combining this observation with
\eqref{eq:orbital-other-bad-ram},
\eqref{eq:Bessel-other-bad-ram},
and
Lemma \ref{lem:Y-facts}
yields the desired formula.
\end{proof}

\begin{thm}
\label{thm:that-bad-ram}
If
\(\depth(X^*) + \depth(Y) = 0\),
and
\(\widetilde Y\) is a stable conjugate of \(Y\)
that lies in a torus with \(X^*\),
then
\begin{align*}
M^G_{X^*}(Y) ={}
& \tfrac1 2 q^{-(h + 1)}\abs{D_\gg(Y)}^{-1/2}\times{} \\
& \qquad\Bigl(
\gamma_\Phi(X^*, Y)
\sum_{\sigma \in W(\bG, \bT_\varpi)}
	\Phi(
		\langle\Ad^*(\sigma)X^*, \widetilde Y\rangle
	) +{} \\
& \qquad\qquad q^{-1/2}\sum_{\substack{
	Z \in (\ttt_\varpi)_{r:r{+}} \\
	Z \ne \pm\widetilde Y
}}
	\Phi(\langle X^*, Z\rangle)
	\sgn_\varpi(Y^2 - Z^2)\Bigr),
\end{align*}
where \(\gamma_\Phi(X^*, Y)\) is as in
Definition \ref{defn:Wald-i}.
\end{thm}

\begin{proof}
Implicit in the statement is the hypothesis that
\(\ttt = \ttt_{\theta'}\) is stably conjugate to
\(\ttt_\varpi\), so that, by \S\ref{sec:tori}, we have
\(\theta' = x^2\varpi\) for some \(x \in \pint\mult\).
The proof proceeds much as in Proposition
\ref{thm:other-bad-ram}.

By
\eqref{eq:ord-s-bad-ram}
and
Corollary \ref{cor:Bessel-twist},
since \(\ord(u) = \ord(v)\),
\eqref{eq:third-orbital-ram} becomes
\begin{equation}
\tag{$*$}
\label{eq:orbital-that-bad-ram}
M^G_{X^*}(Y)
= \abs s^{-1/2}q^{-(h + 1)/2}
\Bessel_{\nu^{1/2}}^\varpi(u, v).
\end{equation}
By \eqref{eq:uv}, we may take the square root \(w\) of \(u v\)
to be \(w = \varpi^{-h}x s\).

Combining
Corollary \ref{cor:Bessel-Kloost}
with
\eqref{eq:uv},
\eqref{eq:ord-u},
and
\eqref{eq:ord-s-bad-ram}
gives
\begin{align*}
\Bessel_{\nu^\alpha}(u, v)
={} &
q\inv\sum_{\substack{
	c \in \pp\inv/\pint \\
	c \ne \pm\varpi^{-h}x s
}}
	\scPhi_{\varpi^{h + 1}}(2c)
	\sgn_\varpi(c^2 - (\varpi^{-h}x s)^2) \\
={} &
q\inv\sum_{\substack{
	c \in \pp^{h - 1}/\pp^h \\
	c \ne \pm x s
}}
	\Phi(2\beta\varpi c)
	\sgn_\varpi(c^2 - x^2 s^2) \\
={} &
q^{-(h + 2)/2}\abs{s\theta'}^{-1/2}
\sum_{\substack{
	c \in \pp^{h - 1}/\pp^h \\
	c \ne \pm x s
}}
	\Phi(2\beta\varpi c)
	\sgn_\varpi(c^2 - x^2 s^2).
\end{align*}
Note that
\(Y^2 = s^2\theta' = x^2 s^2\varpi\),
and that
\[
\widetilde Y \ldef x s\sqrt\varpi
= \Ad\begin{pmatrix}
\sqrt x & 0           \\
0       & \sqrt x\inv
\end{pmatrix}Y
\]
is a stable conjugate of \(Y\) that lies in
\(\ttt_\varpi\).
(Here, \(\sqrt\varpi\) is an element of \(\gg\),
but \(\sqrt x\) is an element of an extension field of
\(\field\).)
As in Theorem \ref{thm:other-bad-ram},
if \(Z = c\dotm\sqrt\varpi\), then
\(\langle X^*, Z\rangle = 2\beta\varpi c\)
and
\(\sgn_\varpi(c^2 - x^2 s^2)
= \sgn_\varpi(Y^2 - Z^2)\).
That is, upon using again the bijection
\(\pp^{h - 1}/\pp^h \to (\ttt_\varpi)_{r:r{+}}\)
given by
\(c \mapsto c\dotm\sqrt\varpi\),
we obtain
\begin{equation}
\tag{${**}_1$}
\label{eq:Bessel-1-bad-ram}
\Bessel_{\nu^{1/2}}(u, v)
= q^{-(h + 2)/2}\abs{s\theta'}^{-1/2}
\sum_{\substack{
	Z \in (\ttt_\varpi)_{r:r{+}} \\
	Z \ne 0, \pm\widetilde Y
}}
	\Phi(\langle X^*, Z\rangle)
	\sgn_\varpi(Y^2 - Z^2).
\end{equation}
Similarly, combining Corollary \ref{cor:Bessel-Kloost} with
\eqref{eq:ord-u},
Lemma \ref{lem:Y-facts},
and
\eqref{eq:sgn-and-G-ram}
gives
\begin{equation}
\tag{${**}_\varpi$}
\label{eq:Bessel-varpi-bad-ram}
\begin{aligned}
\Bessel_{\nu^{1/2}\sgn_\varpi}(u, v)
& {}= q^{-1/2}\gammaram(s)\bigl(
	\Phi(2\beta\varpi x s) +
	\Phi(-2\beta\varpi x s)
\bigr) \\
& {}= q^{-(h + 1)/2}\abs{s\theta'}^{-1/2}\gammaram(s)
\sum_{\sigma \in W(\bG, \bT)}
	\Phi(
		\langle\Ad^*(\sigma)X^*, \widetilde Y\rangle
	).
\end{aligned}
\end{equation}
Combining
\eqref{eq:orbital-that-bad-ram},
\eqref{eq:Bessel-1-bad-ram},
\eqref{eq:Bessel-varpi-bad-ram},
and
Lemma \ref{lem:Y-facts}
gives the desired formula.
\end{proof}

\subsection{Close to zero}
\label{sec:close-ram}

\begin{thm}
\label{thm:close-ram}
If
\(\depth(X^*) + \depth(Y) > 0\),
and
\(X^*\) is ramified,
then let
\(\gamma_\Phi(X^*, Y)\) and \(c_0(X^*)\) be as in
Definitions \ref{defn:Wald-i} and \ref{defn:const},
respectively.
Then
%
\[
M^G_{X^*}(Y)
= c_0(X^*) +
q^{-(h + 1)}\abs{D_\gg(Y)}^{-1/2}\gamma_\Phi(X^*, Y).
\]
\end{thm}

\begin{proof}
By \eqref{eq:m-and-depth-Y}, \(m < 2\).

By Proposition \ref{prop:Bessels} and \eqref{eq:gamma-square},
using Notation \ref{notn:Bessel-abbrev},
\eqref{eq:third-orbital-ram} becomes
\begin{align*}
M^G_{X^*}(Y)
= \tfrac1 2\Bigl(&
\tfrac1 2(1 + \gammaram(s))\bigl[
	Q_3(q^{-1/2}) + \gammaram(s)q\inv;
	1 + \gammaram(s)\inv\sgn_\varpi
\bigr] +{} \\
&\qquad\tfrac1 2(1 - \gammaram(s))\bigl[
	-Q_3(-q^{-1/2}) + \gammaram(s)q\inv;
	(1 - \gammaram(s)\inv\sgn_\varpi)\sgn_\epsilon
\bigr]
\Bigr)(\theta').
\end{align*}
By \eqref{eq:inv-Q3} and the fact that
\[
Q_3(q^{-1/2}) - Q_3(-q^{-1/2})
	= -2T(T^2 + 1)\bigr|_{T = q^{-1/2}}
	= -2q^{-3/2}(q + 1),
\]
we may check case-by-case to see that this simplifies to
\[
M^G_{X^*}(Y)
= \bigl[
	-\tfrac1 2 q^{-3/2}(q + 1);
	1, 0, \gammaram(s), 0
\bigr](\theta').\qedhere
\]
\end{proof}

\section{An integral formula}
\label{sec:orbital-redux}

Remember that all our efforts so far have focussed on the
computation of the function \(M^G_{X^*}\) of Definition
\ref{defn:mock-mu}, whereas we are really interested in the
function \(\hat\mu^G_{X^*}\) of Notation \ref{notn:mu-hat}.
We are now in a position to show that they are actually
equal.

\begin{lemma}
\label{lem:Jacobian}
If \(f \in L^1(G)\), then
\[
\int_G f(g)dg
= \int_{\field_\theta\mult} \int_\field
	f\bigl(\varphi_\theta\inv(\alpha, t)\bigr)
\textup dt\,\textup d\mult\alpha.
\]
\end{lemma}

In the preceding lemma, \(\textup dg\), \(\textup dt\), and
\(\textup d\mult\alpha\) are Haar measures on the obvious
groups.  Given any two of them, the third can be chosen
so that the identity is satisfied.  Since
Definition \ref{defn:mu} requires a measure on
\(G/C_G(X^*)\), not on \(G\), we do not spend much time here
worrying about normalisations (although a specific one is
used in the proof).

\begin{proof}
With respect to the co-ordinate charts
\((a, b, c) \mapsto \begin{smallpmatrix}
a & b          \\
c & (1 + bc)/a
\end{smallpmatrix}\) (for \(a \ne 0\)) on \(G\)
and
\((a, b, t) \mapsto (a + b\sqrt\theta, t)\)
on \(\field_\theta\mult \times \field\),
the Jacobian of \(\varphi_\theta\) at
\(g = \begin{smallpmatrix}
a & b \\
c & d
\end{smallpmatrix}\) (with \(a \ne 0\))
is \(a\inv \Norm_\theta(\alpha)\),
where \(\varphi_\theta(g) = (\alpha, t)\).

In particular, the Haar measure
\[
\abs a\inv\textup da\,\textup db\,\textup dc
\]
on \(G\) is carried to the measure
\[
\abs{\Norm_\theta(a + b\sqrt\theta)}\inv
	\textup da\,\textup db\,\textup dt
= \abs{\Norm_\theta(\alpha)}\inv
	\textup d\alpha\,\textup dt
= \textup d\mult\alpha\,\textup dt
\]
on \(\field_\theta\mult \times \field\),
as desired.
\end{proof}

\begin{prop}
\label{prop:orbital-integral-integral}
If \(X^* \in \gg^*\) and \(Y \in \gg\) are regular and
semisimple, then
\[
\hat\mu^G_{X^*}(Y) = M^G_{X^*}(Y),
\]
where \(M^G_{X^*}\) is as in Definition \ref{defn:mock-mu},
and the Haar measure \(\textup d\dot g\) on \(G/C_G(X^*)\)
of Notation \ref{notn:orbit-measure}
is normalised so that
\[
\meas_{\textup d\dot g}(\dot K) = \begin{cases}
q\inv(q + 1), &
	\text{\(X^*\) split} \\
q\inv(q - 1), &
	\text{\(X^*\) unramified} \\
\tfrac1 2 q^{-2}(q^2 - 1), &
	\text{\(X^*\) ramified,}
\end{cases}
\]
where \(\dot K\) is the image in \(G/C_G(X^*)\) of
\(\SL_2(\pint)\).
\end{prop}

\begin{proof}
We will maintain Notation \ref{notn:X*}.
In particular, \(X^* \in \ttt_\theta^*\).

By the explicit formul\ae\ of the previous sections
(specifically, Theorems
\ref{thm:vanish-spun}, \ref{thm:shallow-spun},
\ref{thm:close-spun},
\ref{thm:vanish-ram}, \ref{thm:shallow-ram},
\ref{thm:other-bad-ram}, \ref{thm:that-bad-ram},
and
\ref{thm:close-ram}),
\(M^G_{X^*} \in C^\infty(\gg\rss)\).
This result plays the role of
\cite{adler-debacker:mk-theory}*{Corollary A.3.4};
we now imitate the proof of Theorem A.1.2 \loccitthendot.

If \(f \in C_c(\gg\rss)\), then there is a lattice
\(\mc L \subseteq \gg\) such that
\(f\dotm M^G_{X^*}\) is invariant under translation by \mc L.
Then
\[
\int_\gg f(Y)M^G_{X^*}(Y)\textup dY
= \meas_{\textup dY}(\mc L)\sum_{Y \in \gg/\mc L}
	f(Y)\dotm\Pint_{\field_\theta\mult/C_\theta} \Pint_\field
		\Phi(\langle X^*, Y\rangle_{\alpha, t})
	\textup dt\,\textup d\mult\dot\alpha.
\]
Since the sum is finitely supported, we may bring it
inside the integral.
By \eqref{eq:hilarious-trace} and
Definition \ref{defn:mu-hat},
\begin{equation}
\tag{$*$}
\label{eq:funny-integrated}
\begin{aligned}
& \int_\gg f(Y)M^G_{X^*}(Y)\textup dY \\
& \qquad= \Pint_{\field_\theta\mult/C_\theta} \Pint_\field
	\meas_{\textup dY}(\mc L)\sum_{Y \in \gg/\mc L}
		f(Y)\Phi\bigl(\langle
			\Ad^*(\varphi_\theta\inv(\alpha, t))X^*,
			Y
		\rangle\bigr)
\textup dt\,\textup d\mult\dot\alpha \\
& \qquad= \Pint_{\field_\theta\mult/C_\theta} \Pint_\field
	\int_\gg f(Y)\Phi\bigl(\langle
		\Ad^*(\varphi_\theta\inv(\alpha, t))X^*,
		Y
	\rangle\bigr)\textup dY\,
\textup dt\,\textup d\mult\dot\alpha \\
& \qquad= \Pint_{\field_\theta\mult/C_\theta} \Pint_\field
	\hat f\bigl(
		\Ad^*(\varphi_\theta\inv(\alpha, t))X^*
	\bigr)
\textup dt\,\textup d\mult\dot\alpha,
\end{aligned}
\end{equation}
where \(\varphi_\theta\) is as in
Definition \ref{defn:sl2-as-fields}.

On the other hand, again by Definition \ref{defn:mu-hat},
\begin{align*}
\hat\mu^G_{X^*}(f) \ldef \mu^G_{X^*}(\hat f)
& = \int_{G/T_\theta}
	\hat f\bigl(\Ad^*(g)X^*\bigr)
\textup d\dot g \\
& = \int_{\ol U\backslash G/T_\theta} \int_{\ol U}
	\hat f\bigl(\Ad^*(\ol u g)X^*\bigr)
\textup d\ol u\,\textup d\ddot g,
\end{align*}
where \(\ol U = \set{\begin{smallpmatrix}
1 & 0 \\
b & 1
\end{smallpmatrix}}{b \in \field}\).
By Lemmata \ref{lem:Jacobian} and \ref{lem:torus-acts},
and \eqref{eq:funny-integrated},
if \(\textup d\dot g\) is properly normalised, then
\[
\hat\mu^G_{X^*}(f)
= \int_{\field_\theta\mult/C_\theta} \int_\field
	\hat f\bigl(
		\Ad^*(\varphi_\theta\inv(\alpha, t))X^*
	\bigr)
\textup dt\,\textup d\mult\dot\alpha
= \int_\gg f(Y)M^G_{X^*}(Y)\textup dY.
\]

It remains only to compute the normalisation of
\(\textup d\dot g\).  We do so case-by-case.
If \(X^*\) is split, so that we may take \(\theta = 1\),
then the image under \(\varphi_1\) of
\[
(1 + \pp_1) \times \pp
\subseteq \field_1\mult \times \field
\]
is precisely the kernel \(K_+\) of the
(component-wise) reduction map
\(\SL_2(\pint) \to \SL_2(\resfld)\).
Here, we have written
\(1 + \pp_1 = \set{(a, b) \in \field_1}{
	a \in 1 + \pp, b \in \pp
}\).
Thus,
\[
(1 + \pp_1)C_1/C_1 \times \pp
\overset\sim\rightarrow
K_+T_1/T_1.
\]
Now
\(\Norm_1 : 1 + \pp_1 \to 1 + \pp\) is surjective,
so, by Definitions \ref{defn:k-Haar} and \ref{defn:mock-mu},
the measure (in \(\field_1/C_1 \times \field\)) of the
domain is
\[
\meas_{\textup d\mult x}(1 + \pp)\dotm
	\meas_{\textup dx}(\pp)
= q^{-2}.
\]
Since \(\dot K = \SL_2(\pint)T_1/T_1\) is tiled by
\[
\bigl[ \SL_2(\pint)T_1 \colon K_+T_1 \bigr]
= \bigl[ \SL_2(\pint) \colon K_+(T_1 \cap \SL_2(\pint)) \bigr]
= \bigl[ \SL_2(\resfld) \colon \ms T_1(\resfld) \bigr]
= q(q + 1)
\]
copies of \(K_+T_1/T_1\),
where \(\ms T_1 \ldef \set{\begin{smallpmatrix}
a & b \\
b & a
\end{smallpmatrix}}{a^2 - b^2 = 1}\),
we see that, in this case,
\(\textup d\dot g\) assigns \(\dot K\) measure
\(q^{-2}\dotm q(q + 1) = q\inv(q + 1)\).

The remaining cases are easier, since
\(C_\theta\) is contained in the ring \(\pint_\theta\)
of integers in \(\field_\theta\),
and (for our choices of \(\theta\))
\(T_\theta\) is contained in \(\SL_2(\pint)\).
If \(X^*\) is unramified, so that we may take
\(\theta = \epsilon\), then the image under
\(\varphi_\epsilon\) of \(\pint_\epsilon\mult \times \pint\)
is precisely \(\SL_2(\pint)\).
Since
\(\Norm_\epsilon : \pint_\epsilon\mult \to \pint\mult\)
is surjective,
we see that, in this case,
\(\textup d\dot g\) assigns \(\dot K\) measure
\(\meas_{\textup d\mult x}(\pint_\epsilon\mult)
	\dotm\meas_{\textup dx}(\pint) = q\inv(q - 1)\).

If \(X^*\) is ramified, so that we may take
\(\theta = \varpi\), then the image under
\(\varphi_\varpi\) of \(\pint_\varpi\mult \times \pp\) is
precisely the Iwahori subgroup \mc I,
i.e., the pre-image in \(\SL_2(\pint)\) of
\(\ms B(\resfld) \ldef \set{\begin{smallpmatrix}
a & b     \\
0 & a\inv
\end{smallpmatrix}}{a \in \resfld\mult, b \in \resfld}\)
under the reduction map
\(\SL_2(\pint) \to \SL_2(\resfld)\).
Since
\(\Norm_\varpi : \pint_\varpi\mult \to \pint\mult\)
has co-kernel of order \(2\),
we see that, in this case,
\(\textup d\dot g\) assigns \(\dot K\) measure
\[
\tfrac1 2\meas_{\textup d\mult x}(\pint\mult)\dotm
	\meas_{\textup dx}(\pp)\dotm
\bigl[ \SL_2(\resfld) \colon \ms B \bigr]
= \tfrac1 2q^{-2}(q^2 - 1).\qedhere
\]
\end{proof}

In particular, all the results we have proven for
\(M^G_{X^*}\) are actually results about
\(\hat\mu^G_{X^*}\).  We close by summarising some results
that can be stated in a fairly uniform fashion
(i.e., mostly independent of the `type' of \(X^*\), in the
sense of Definition \ref{defn:split-un-or-ram}).
This theorem does \emph{not} cover everything we have shown
about Fourier transforms of semisimple orbital integrals
(in particular, it says nothing about the behaviour of
ramified orbital integrals on the `bad shell', as in
\S\ref{sec:bad-ram});
for that, the reader should refer to the detailed results
of \S\S\ref{sec:orbital-spun}--\ref{sec:orbital-ram}.

\begin{thm}
\label{thm:uniform}
If \(\depth(X^*) + \depth(Y) < 0\)
(or \(X^*\) is split or unramified and
\(\depth(X^*) + \depth(Y) \le 0\)),
then
\[
\hat\mu^G_{X^*}(Y)
= q^{-(\scdepth + 1)}\abs{D_\gg(Y)}^{-1/2}
\gamma_\Phi(X^*, Y)
\sum_{\sigma \in W(G, T)}
	\Phi(\langle\Ad^*(\sigma)X^*, Y\rangle)
\]
if \(X^*\) and \(Y\) lie in a common torus \bT
(with \(T = \bT(\field)\)),
and
\[
\hat\mu^G_{X^*}(Y) = 0
\]
if \(X^*\) and \(Y\) do not lie in \(G\)-conjugate tori.
Here, \(\scdepth\) is as in
Notation \ref{notn:depth-and-scdepth},
and \(\gamma_\Phi(X^*, Y)\) is as in
Definition \ref{defn:Wald-i}.

If \(\depth(X^*) + \depth(Y) > 0\),
then
\[
\hat\mu^G_{X^*}(Y)
= c_0(X^*) +
q^{-(\scdepth + 1)}\abs{D_\gg(Y)}^{-1/2}\gamma_\Phi(X^*, Y).
\]
Here,
\(\gamma_\Phi(X^*, Y)\) and \(c_0(X^*)\)
are as in
Definitions \ref{defn:Wald-i} and \ref{defn:const},
respectively.
\end{thm}

\begin{proof}
This is an amalgamation of parts of
Theorems
	\ref{thm:vanish-spun}, \ref{thm:shallow-spun},
	\ref{thm:close-spun},
	\ref{thm:vanish-ram}, \ref{thm:shallow-ram},
and \ref{thm:close-ram},
and Proposition \ref{prop:orbital-integral-integral}.
\end{proof}

\end{document}